\documentclass[11pt]{article}
 
\usepackage[margin=0.5in]{geometry} 
\usepackage{amsmath,amsthm,amssymb}
\usepackage{mathrsfs}
\usepackage[T1]{fontenc}
\usepackage[english]{babel}
\usepackage{graphicx}
\usepackage{color} 

\usepackage{hyperref}
\usepackage{afterpage}

\usepackage{color}

\usepackage{mathtools}
\usepackage{atbegshi}
\AtBeginDocument{\AtBeginShipoutNext{\AtBeginShipoutDiscard}}
\usepackage{pdfpages} 
\usepackage{longtable} 
\usepackage{amsmath,amsthm,amsfonts,amssymb,relsize,bm} 
\usepackage{xparse,xcolor} 
\usepackage{graphicx,float} 
\usepackage{etoolbox}
\usepackage[shortlabels]{enumitem} 
\usepackage{tikz-cd} 
\usepackage{cite} 
\usepackage[normalem]{ulem}
\usepackage{comment}
\usepackage{tocloft}

\numberwithin{equation}{section}

\theoremstyle{plain}
\newtheorem{theorem}{Theorem}[section]

\newtheorem*{theorem*}{Theorem}
\newtheorem{lemma}{Lemma}[section]
\newtheorem{corollary}{Corollary}[section]

\theoremstyle{definition}

\newtheorem*{definition*}{Definition}

\newtheorem{remark}{Remark}[section]

\begin{document}
 
 \title{Analogues of a formula of Ferrar: what I have learned from Semyon Yakubovich}

\author{{Pedro Ribeiro}} 
\thanks{


Natixis in Portugal (Groupe BPCE), Rua de Santos Pousada 220, 4000-478 Porto. 

\,\,\,\,\,E-mail of the author: pedromanelribeiro1812@gmail.com}

\date{}
 
\maketitle

\begin{center}
\textit{Dedicated to Semyon Yakubovich, on the occasion of his 25 years in Portugal}
\end{center}

\begin{abstract} 
W. L. Ferrar seems to have been the first mathematician to clearly draw a connection
between the functional aspects of a summation formula and the behavior of
the Dirichlet series underlying it. 
Taking a formula due to him as a starting point, I will describe some new generalizations of Ferrar's formulas and how these were actually
obtained after learning a great deal from Semyon.
I also present a very concise overview of the underlying theory of summation formulas and how the Mellin transform has been the link between mine and Professor Yakubovich's interests. 
\end{abstract}

\pagenumbering{arabic}

\section{Introduction}

Ever since Riemann established the connection between the analytic structure of
an ordinary infinite series and the distribution of the prime numbers, the investigation of the zeros of $\zeta(s)$ has emerged as the main motivation
driving some of the major breakthroughs in Number Theory.

\bigskip{}

One may naturally ask what Riemann’s original motivation was. Judging from his 1859 memoir, it seems clear to most of us that he
really intended to give a formal proof of the \textit{Prime Number Theorem}. However, Riemann gives special attention to the proof of the functional equation for $\zeta(s)$ (indicating two different proofs of the same result), which may suggest that he was more interested in the formal aspects of the function itself. 
Regardless of his true motivations, Riemann undoubtedly understood
that, in order to fully grasp the infinite series
\begin{equation}
\zeta(s)=\sum_{n=1}^{\infty}\frac{1}{n^{s}},\label{Riemann zeta function}
\end{equation}
the argument of the associated function would have to be complex.
Since (\ref{Riemann zeta function}) only makes sense in the half
plane $\text{Re}(s)>1$, he would have to find a way to reflect the
behavior of (\ref{Riemann zeta function}) into the complementary
region $\text{Re}(s)<1$.

The ingenious way Riemann found to do this was through the functional equation
\begin{equation}
\pi^{-\frac{s}{2}}\Gamma\left(\frac{s}{2}\right)\zeta(s)=\pi^{-\frac{1-s}{2}}\Gamma\left(\frac{1-s}{2}\right)\zeta\left(1-s\right),\label{Functional equation Riemann!}
\end{equation}
which he proved by firstly establishing the integral formula
\begin{equation}
\intop_{0}^{\infty}\frac{x^{s-1}}{e^{x}-1}dx=\Gamma(s)\,\zeta(s),\,\,\,\,\,\,\text{Re}(s)>1. \label{riemann formula int}
\end{equation}

Despite his truly magnificent contributions to the study of the function now bearing his name, Riemann stood on the
shoulders of Euler. Euler's most famous contribution to the theory of the infinite series (\ref{Riemann zeta function}) came in 1734, when he proved the astonishing formula
\begin{equation}
\zeta(2n)=\frac{(-1)^{n-1}(2\pi)^{2n}}{2(2n)!}B_{2n},\label{zeta (2n) formula}
\end{equation}
where $B_{2n}$ denotes the $2n^{\text{th}}$ Bernoulli number. A
special case of Euler's formula (\ref{zeta (2n) formula}), with particular
historical importance, is the Basel identity
\begin{equation}
\zeta(2)=\frac{\pi^{2}}{6},\label{Basel Intro}
\end{equation}
for which several proofs are known today. 

\bigskip{}

In his book on the Riemann zeta function [\cite{edwards}, p. 12],
Prof. H. M. Edwards offers a different explanation on Riemann's motivations
to get the functional equation (\ref{Functional equation Riemann!}).
Instead of being solely interested in prime numbers, Edwards argues
that the first motivation behind Riemann's derivation of (\ref{Functional equation Riemann!})
was the necessity of deducing a new proof of Euler's formula (\ref{zeta (2n) formula}). As he aptly writes

\begin{center}
``\textit{ There is no easy way to deduce this famous
formula of Euler\textquoteright s {[}eq. (\ref{zeta (2n) formula})
here{]} from Riemann\textquoteright s integral formula for $\zeta(s)$
{[}equation (\ref{riemann formula int}) here{]} and it may
well have been this problem of deriving {[}eq. (\ref{zeta (2n) formula})
here{]} anew which led Riemann to the discovery of the functional
equation of the zeta function }''. 
\end{center}

\bigskip{}

I still remember how captivated I felt when I first came across this passage in 2018. At the time, Euler’s formula for $\zeta(2n)$, (\ref{zeta (2n) formula}), was my constant mathematical companion. Although I was only 20 then, I had already written and published two papers dedicated to that identity \cite{ribeiro_gazette, ribeiro_monthly}. Discovering that, in some corner of mathematical history, Riemann himself might have shared this fascination deeply resonated with me. It was this connection (however speculative) that inspired me to seriously engage with the deeper and more demanding theory surrounding the Riemann zeta function.

\bigskip{}

Being so interested in the Riemann zeta function, the next natural
step was to study its analytic continuation and the behavior of its
zeros. At this point, I encountered one of the most astonishing results in Mathematics, known as \textit{Hardy's Theorem}. In a brief but remarkable note \cite{hardy_note}, Hardy demonstrated that infinitely many zeros of $\zeta(s)$ lie on the critical line, that is, on $\text{Re}(s) = \frac{1}{2}$. 

\bigskip{}
Three major ingredients were crucial in Hardy's proof:
\begin{enumerate}
\item The function 
\[
\eta(s):=\pi^{-\frac{s}{2}}\Gamma\left(\frac{s}{2}\right)\zeta(s)
\]
is real-valued when $\text{Re}(s)=\frac{1}{2}$.
\item The following Fourier transform is valid
\begin{equation}
\intop_{-\infty}^{\infty}\eta\left(\frac{1}{2}+it\right)\,e^{\omega t}dt=-4\pi\,\cos\left(\frac{\omega}{2}\right)+2\pi\,e^{\frac{i\omega}{2}}\theta\left(e^{2i\omega}\right),\,\,\,0<\omega<\frac{\pi}{4},\label{Integral Hardy}
\end{equation}
where $\theta(x)$ is the Jacobi $\theta-$function,
\[
\theta(x)=\sum_{n\in\mathbb{Z}}e^{-\pi n^{2}x},\,\,\,\,\,\text{Re}(x)>0.
\]
\item The Jacobi $\theta-$function obeys the transformation formula
\begin{equation}
\theta(x)=\frac{1}{\sqrt{x}}\theta\left(\frac{1}{x}\right),\,\,\,\,\text{Re}(x)>0.\label{theta formula intro}
\end{equation}
\end{enumerate}
Each of the items listed above would play a pivotal role in my academic formation, as it was through the study of Hardy’s original proof that I first encountered many of the techniques I regularly used during my PhD. Item 1 introduced me to the concept of analytic continuation; items 2 and 3, in turn, opened the door to Fourier Analysis and the theory of summation formulas, respectively.

\bigskip{}

It was around this time, while I was first studying Hardy’s theorem, that I met Semyon. After completing my Bachelor’s degree in Physics (and already being familiar with several special functions) I realized that I wanted to pursue Number Theory and Analysis more seriously. At the same time, I felt that my mathematical maturity needed significant refinement, which led me to enroll in the Bachelor’s degree in Pure Mathematics at the University of Porto.
I met Semyon in a course on Linear Analysis, an introductory course in Functional Analysis. I vividly remember telling him, after a lecture on Sturm–Liouville problems, that one could construct suitable Dirichlet boundary conditions to obtain a new proof of the Basel identity (\ref{Basel Intro}). He replied \textit{“But it is much easier to use the Fourier series of 
$f(x)=x.$”} I then added, \textit{“Or maybe one could use the Poisson summation formula.”}

At the mention of summation formulas - an idea that had already captivated me, and one so deeply connected to Riemann’s work - his interest was immediate. \textit{“Are you interested in summation formulas?”} he asked. \textit{“Then you should take a look at my paper!”}\footnote{
This conversation was held in Portuguese; the English translation provided here captures the essence of our exchange. The paper Semyon was mentioning was his paper on the M\"untz class of functions, \cite{new_summation_formulas}.}

\bigskip{}

I readily took a look at Semyon’s paper, and soon after at several of the works he produced during the period 2010–2015. The present paper, like almost every work I have written since, may be seen as a corollary of that same decision made in June 2018.

The purpose of this paper, written in tribute to him, is to study a generalization of a well-known summation formula due to W. L. Ferrar.
But what is a summation formula, one may ask? Informally, a \textit{summation formula} may be described as a relation between an infinite series of the type
\begin{equation}
\sum_{n=0}^{\infty}a(n)\,f\left(\lambda_{n}\right),\,\,\,\,\,\,0<\lambda_{n}\nearrow\infty\label{Intro - 1}
\end{equation}
with a corresponding series of the form 
\begin{equation}
\sum_{n=0}^{\infty}b(n)\,\intop_{0}^{\infty}f(y)\,K\left(\mu_{n}y\right)\,dy,\,\,\,\,\,\,0<\mu_{n}\nearrow\infty,\label{Intro - 2}
\end{equation}
where $a(n)$ and $b(n)$ are suitable arithmetical functions allowing the convergence of the infinite series (\ref{Intro - 1})
and (\ref{Intro - 2}) and $K(x)$ is a kernel, usually involving special functions, that only depends on $a(n)$ and $b(n)$. 

\medskip{}

Of course, the most famous summation formula in the form
\begin{equation}
\sum_{n=0}^{\infty}a(n)\,f(\lambda_{n})=\sum_{n=0}^{\infty}b(n)\,\intop_{0}^{\infty}f(y)\,K(\mu_{n}\,y)\,dy\label{Intro - 3}
\end{equation}
is due to Poisson and achieved when $\lambda_{n}=\mu_{n}=n$, $a(0)=\frac{1}{2}$,
$b(0)=1$ and $a(n)=1,\,b(n)=2$. With these substitutions, one
obtains the recognizable identity, valid for the elementary kernel $K(x)=\cos(2\pi x)$, 
\begin{equation}
\frac{1}{2}f(0)+\sum_{n=1}^{\infty}f(n)=\intop_{0}^{\infty}f(y)\,dy+2\sum_{n=1}^{\infty}\intop_{0}^{\infty}f(y)\,\cos(2\pi ny)\,dy.\label{Intro - 4}
\end{equation}
Ever since Dirichlet's proof of Poisson's formula \cite{dirichletserisV}, it
is also known that, if $f$ is a continuous function and of bounded
variation on $[a,b]$, a finite version of
(\ref{Intro - 4}) holds in the following form
\begin{equation}
\sum_{n=a}^{b}{}^{\prime}\,f(n)=\intop_{a}^{b}f(y)\,dy+2\,\sum_{n=1}^{\infty}\intop_{a}^{b}f(y)\,\cos(2\pi n\,y)\,dy,\label{Intro - 5}
\end{equation}
where the prime on the summation sign at the left indicates that, if $a$ or $b$ are integers, then only $\frac{1}{2}\,f(a)$
or $\frac{1}{2}\,f(b)$ contributes to the sum. 

Imposing conditions for the validity of (\ref{Intro - 4}) is a relatively studied subject as there are several ways of approaching Poisson's formula through different methods. The first one was given by Dirichlet in the form (\ref{Intro - 5}) \cite{Davenport} and it consists in seeing (\ref{Intro - 4}) as an expansion of a periodic function into Fourier series. 
Another approach, due to Cauchy \cite{guinand_poisson}, uses Complex Analysis and consists in assuming that $f(x)$ extends to $\mathbb{C}$ as an analytic function. This assumption means that the meromorphic function $g(z)=\pi \cot(\pi z)\,f(z)$ has simple poles located at the integers and an application of the Residue Theorem \cite{stein} yields the summation formula (\ref{Intro - 4}).\footnote{Stronger versions and proofs of Poisson's summation formula appeared in the papers of Mordell and Wilton \cite{Mordell_N.T., Mordell_Poisson, Wilton_Poisson}. As we shall see later on this introduction, Dixon and Ferrar \cite{dixon_ferrar_Poisson_Voronoi} were also prolific on this topic. }

One of the most famous corollaries of Poisson's summation formula (\ref{Intro - 5}) is the beautiful formula
\begin{equation}
\sum_{n\in \mathbb{Z}}e^{-n^2\pi x}=\frac{1}{\sqrt x}\,\sum_{n\in \mathbb{Z}} e^{-n^2 \pi/x},\,\,\,\,\,\,\text{Re}(x)>0, \label{transf formula intro theta book}
\end{equation}
known as \textit{transformation formula for Jacobi's $\theta-$function.}\footnote{There are several proofs of this formula, but the usual temptation is to prove (\ref{transf formula intro theta book}) via Poisson's summation formula (\ref{Intro - 4}), with $f(t)=e^{-\pi x t^2}$, $\text{Re}(x)>0$. In 1840, Cauchy applied his work on the Poisson summation formula and the theta transformation formula to evaluate Gauss' sums \cite{gauss_sums_berndt}.} Perhaps the most meaningful application of the theory of summation formulas, in
particular of the identity (\ref{transf formula intro theta book}), is due to the work of Bernhard Riemann himself. In his famous memoir \cite{riemann}, Riemann employed the theta transformation formula to
prove the functional equation for his famous zeta function,
\begin{equation}
\pi^{-\frac{s}{2}}\Gamma\left(\frac{s}{2}\right)\zeta(s)=\pi^{-\frac{1-s}{2}}\Gamma\left(\frac{1-s}{2}\right)\zeta\left(1-s\right).\label{Riemann equation introduction CRC book}
\end{equation}
As we have already noted, Hardy \cite{hardy_note} also saw that the symmetries of the transformation formula (\ref{transf formula intro theta book}) could be exploited to show a striking feature that $\zeta(s)$, as well as other Dirichlet series\footnote{as shown by Landau \cite{landau_hardy}, who realized 3 months after the publication of Hardy's note \cite{hardy_note}, that his theta function method could be extended to Dirichlet $L-$functions as well as Epstein zeta functions.}, possesses: an infinitude of zeros at its critical line, this is, $\text{Re}(s)=\frac{1}{2}$.
This result of Hardy, together with the first zero density estimates
due to Bohr and Landau \cite{bohr_landau}, are, from a historical perspective, the first pieces of
evidence towards the Riemann hypothesis.\footnote{The reader should take this assertion with a huge grain of salt. In
this passage of our introduction we are essentially paraphrasing H.
M. Edwards \cite{edwards} who writes that the clustering of the zeros near
the critical line $\text{Re}(s)=\frac{1}{2}$ (i.e., the Bohr-Landau
theorem) is the best evidence for the Riemann hypothesis. However,
as shown by Levinson \cite{levinson_avalues}, for any complex $a$, the zeros
of the function $\zeta(s)-a$ also cluster around the critical line
and so, in this sense, the case where $a=0$ might not be special at all.}

\bigskip{}

Poisson's formula and its consequence still constitute, however, very special cases and, from a nineteenth century perspective, were actually treated under very strict conditions. One may ask, thus, under which more general conditions a summation formula of the type (\ref{Intro - 3}) holds and how the kernel $K(x)$, present in the integral transform on the right-hand side of it, depends on the coefficients $a(n)$ and $b(n)$. 
 
\medskip{}

In 1904, Vorono\"i \cite{dirichletserisV} made the following conjecture/question, regarded as the first systematic attempt at generalizing the Poisson summation formula (\ref{Intro - 5}): if one takes $a(n)=b(n)$
and assume that $f(x)$ is continuous on $[a,b]$ with only a finite number of
maxima and minima there, can one prove that there always exist analytic functions
$\delta(x)$ and $K(x)$, depending only on $a(n)$, such that a generalized
form of (\ref{Intro - 5}),  
\begin{equation}
\sum_{n=a}^{b}{}^{\prime}a(n)\,f(n)=\intop_{a}^{b}f(y)\,\delta(y)\,dy+\sum_{n=1}^{\infty}a(n)\,\intop_{a}^{b}f(y)\,K(ny)\,dy,\label{Intro - 6}
\end{equation}
holds?
After considerable efforts, Vorono\"i was able to prove that his conjecture
is true when $a(n)$ is the divisor
function, $d(n)=\sum_{d|n}1$. In this case, the associated kernels, $\delta(x)$
and $K(x)$, are respectively given by a residual term\footnote{with ''residual term'' we mean a term whose contribution is essentially due to the meromorphic/residual of the Dirichlet series attached to the arithmetical sequence $a(n)$. In Vorono\"i's case, where $a(n)=d(n)$, the Dirichlet series is $\zeta^2(s)$.} and a combination of Bessel functions of the second kind, this is, 
\[
\delta(x)=\log(x)+2\gamma, \,\,\,\,\,\,\, K(x)=4\,K_{0}\left(4\pi\sqrt{x}\right)-2\pi\,Y_{0}\left(4\pi\sqrt{x}\right),
\]
where $\gamma$ denotes the Euler-Mascheroni constant.
This special case of (\ref{Intro - 6}) when $a(n)=d(n)$ offers the interesting relation 
\begin{equation}
\sum_{n=a}^{b}{}^{\prime}\,d(n)\,f(n)=\intop_{a}^{b}f(y)\,\left(\log(y)+2\gamma\right)\,dy+\sum_{n=1}^{\infty}d(n)\,\intop_{a}^{b}f(y)\,\left[4\,K_{0}(4\pi\sqrt{ny})-2\pi\,Y_{0}(4\pi\sqrt{ny})\right]\,dy,\label{Intro - 7}
\end{equation}
usually known as Vorono\"i's summation formula. 

\medskip{}
Formula (\ref{Intro - 7}) was later proved by Koshliakov (1928) with
the assumption of $f(x)$ being an analytic function, thus being regarded as an extension of Cauchy's proof of Poisson's formula.\footnote{This version of Vorono\"i's summation formula due to Koshliakov was instrumental in the first rigorous proof of Popov's formula given in \cite{berndt_popov}.} After Koshliakov's proof, several
other arguments began to appear, all of them assuming different conditions
over the function $f$. A. L. Dixon and W. L. Ferrar \cite{dixon_ferrar_lattice} gave a proof of (\ref{Intro - 7})
under the condition that $f\in C^{2}[a,b]$ and Wilton \cite{Wilton_Voronoi} proved
Vorono\"i's formula for functions of bounded variation on $[a,b]$ and
even extended it to $b=\infty$.\footnote{A further extension of Wilton's result to accommodate the case where $a=0$ was later considered by Dixon and Ferrar \cite{dixon_ferrar_Poisson_Voronoi}.} It is universally agreed that the most beautiful example of Vorono\"is formula shows up when the function $f(t)$ in (\ref{Intro - 7}) is taken to be $K_{0}(2\pi zt)$. This result can be explicitly written as
\begin{equation}
\sum_{n=1}^{\infty}d(n)\,K_{0}\left(2\pi nz\right)-\frac{1}{z}\sum_{n=1}^{\infty}d(n)\,K_{0}\left(\frac{2\pi n}{z}\right)
=\frac{1}{4z}\left(\gamma-\log\left(4\pi z\right)\right)-\frac{1}{4}\left(\gamma-\log\left(\frac{4\pi}{z}\right)\right), \,\,\,\,z>0,\label{III-84}
\end{equation}
where, again, $K_{0}(x)$ denotes the modified Bessel function of the second kind. This formula is usually known as Koshliakov's formula, despite the fact that it was discovered by Ramanujan ten years prior \cite{Guinand_Ramanujan}. 
A. L. Dixon and W. L. Ferrar also proved (\ref{III-84}) as a consequence of their variant of Vorono\"i's formula [\cite{dixon_ferrar_Poisson_Voronoi}, pp. 70-71], but Ferrar had earlier employed the Mellin transform to prove (\ref{III-84}) through the functional equation for $\zeta^2(s)$ \cite{ferrar}.\footnote{This alternative argument of Ferrar is the main motivation behind the present paper!} Using a similar
approach to the one employed by Ferrar in his paper on modular relations of the form (\ref{III-84}) \cite{ferrar}, the joint work of F. Oberhettinger and K. L. Soni \cite{relations_equivalent} used Ferrar's idea to express not only (\ref{III-84}), but several other summation formulas with $d(n)$ in (\ref{III-84}) replaced by a more general arithmetical function $a(n)$.  

After proving his conjecture for the arithmetical function $d(n)$,
Vorono\"i also announced a corresponding result for another arithmetical
function, $r_{2}(n)$, which has the role of counting the number of ways in
which a given positive integer $n$ can be expressed as a sum of two
squared integers. Vorono\"i presented the following
formula, analogous to (\ref{Intro - 7}), 
\begin{equation}
\sum_{n=a}^{b}{}^{\prime}\, r_{2}(n)\,f(n)=\pi\,\intop_{a}^{b}f(y)\,dy+\pi\sum_{n=1}^{\infty}r_{2}(n)\,\intop_{a}^{b}f(y)\,J_{0}\left(2\pi\sqrt{ny}\right)\,dy,\label{Intro - 8}
\end{equation}
where $J_{\nu}(x)$ denotes the Bessel function of the first kind. Another
version of (\ref{Intro - 8}) was also established by Sierpi\'nski and
Laudau \cite{dirichletserisV} for functions of bounded variation and an extension
to $b=\infty$, invoking additional conditions on the decay of $f$,
was made by A. L. Dixon and W. L. Ferrar  \cite{dirichletserisV}. As an example of their variant of Vorono\"i's result, Dixon and Ferrar established the following identity [\cite{dixon_ferrar_circle(i)}, p. 51, eq. (3.12)]
\begin{equation}
\frac{a^{\nu/2}\Gamma(\nu+1)}{2\pi^{\nu+1}}\sum_{n=0}^{\infty}\frac{r_{2}(n)}{(n+a)^{\nu+1}}=\sum_{n=0}^{\infty}r_{2}(n)\,n^{\frac{\nu}{2}}\,K_{\nu}\left(2\pi\sqrt{an}\right),\,\,\,\,\,\text{Re}(\nu)>0,\label{III-106}
\end{equation}
which is, in the spirit of the general theory of summation formulas, the first example of a Bessel expansion, as it connects a generalized Dirichlet series (located on the left-hand side of (\ref{III-106})) with a series involving the modified Bessel function, $K_{\nu}(x)$.\footnote{The reader is strongly advised to complement the reading of this introduction with the magnificent survey on summation formulas given in \cite{BDGZ_survey}.} 

\bigskip{}

Even within a quite formal context, the main purpose of Vorono\"i's conjecture was to study the interdependence
between the arithmetical function $a(n)$\footnote{which, recall, fixes the arithmetical function $b(n)$, as Vorono\"i assumes that $a(n)=b(n)$.} and the kernels $\delta(x)$
and $K(x)$: this is the core issue that makes a summation formula relevant. The first step taken towards greater generality addressing this question is due
to Ferrar (1935-1937) [\cite{dirichletserisV}, p. 140], who proved that the kernel $K(x)$
owes its behavior to the functional equations for the Dirichlet series,\footnote{in Ferrar's convention, $\psi(s)$ does not need to be equal to $\phi(s)$, so the arithmetical functions $a(n)$ and $b(n)$ may be assumed to be distinct from one another.}
\begin{equation}
\phi(s)=\sum_{n=1}^{\infty}\frac{a(n)}{\lambda_{n}^{s}},\,\,\,\,\,\,\,\,\psi(s)=\sum_{n=1}^{\infty}\frac{b(n)}{\mu_{n}^{s}},\,\,\,\,\,\text{Re}(s)>\max\{\sigma_{1},\,\sigma_{2}\},\label{Dirichlet series}
\end{equation}
which contain all the data concerning a summation formula of the type (\ref{Intro - 3}). 
Ferrar was evidently the first mathematician to remark that Vorono\"i's
summation formula is equivalent to the functional equation for
$\zeta^{2}(s)$ \cite{ferrar, ferrar_I, ferrar_II}, in the same way as Poisson's summation formula is equivalent to Riemann's functional equation.\footnote{Which in its turn means that the theta transformation formula (\ref{transf formula intro theta book}) is equivalent to the functional equation (\ref{Riemann equation introduction CRC book}). Furthermore, this means that the theta transformation formula is equivalent to Poisson's summation formula applied to a general class of functions. This reveals a \textit{domino property} of summation formulas, a property which states that, in order to prove a general summation formula for a suitable class of functions, it is enough to prove that same summation formula only for a representative example. If this example is strong enough, it can be used to prove the functional equation for one of the Dirichlet series associated to the summation formula we want to prove! Finally, the functional equation for this Dirichlet series will be equivalent to the general, wider version of the summation formula we wanted to prove in the first place! This circle of ideas naturally leads to the first general considerations of summation formulas, started by Bochner, Chandrasekharan and Narasimhan \cite{Bochner_connections, bochner_modular_relations, bochner_chrandrasekharan, arithmetical identities}.}

Ferrar also pointed out one intriguing fact concerning the summation
formulas. He observed, from the point of view of integral transforms \cite{relations_equivalent},
that the kernels lying in each of the summation formulas (\ref{Intro - 7})  and (\ref{Intro - 8}) satisfy the reciprocity relations
\begin{equation}
g(x)=\intop_{0}^{\infty}f(y)\,K(x\,y)\,dy,\,\,\,\,\,\,\,\,f(x)=c\,\intop_{0}^{\infty}g(y)\,K(x\,y)\,dy,
\end{equation}
where $c$ is some normalization constant. The kernel $K(x)$ is usually
called a Fourier kernel, Hankel kernel or even Fourier-Watson kernel \cite{titchmarsh_fourier_integrals}.
The striking connections between the summation formulas involving
the arithmetical functions $a(n)$ and $b(n)$ and their respective
Dirichlet series and integral transforms are one of the main motivations for my collaboration with Semyon first as a Master student and later as his PhD student.  

\bigskip{}

After these early attempts by Ferrar to provide a general method to deal with
summation formulas, S. Chandrasekharan and R. Narasimhan \cite{arithmetical identities} 
took an extensive setting of functional equations (the so called 'Hecke-Type') for the Dirichlet
series $\phi(s)$ and $\psi(s)$ in the form 
\begin{equation}
\Gamma(s)\,\phi(s)=\Gamma(r-s)\,\psi(r-s)\label{Hecke type Functional Equations}
\end{equation}
and used it to deduce several arithmetical identities for the coefficients
$a(n)$ and $b(n)$. Moreover, Chandrasekharan and Narasimhan proved the equivalence between such arithmetical identities and the functional equation (\ref{Hecke type Functional Equations}). One of the most remarkable examples obtained by these authors is the identity involving 
Ramanujan's $\tau-$function [\cite{arithmetical identities}, p. 16, eq. (56)],  
\begin{equation}
\sum_{n=1}^{\infty}\tau(n)\,e^{-x\sqrt{n}}=2^{36}\pi^{\frac{23}{2}}\Gamma\left(\frac{25}{2}\right)\,\sum_{n=1}^{\infty}\frac{\tau(n)}{\left(x^{2}+16\pi^{2}n\right)^{\frac{25}{2}}}, \label{chandra nara intro book}
\end{equation}
valid whenever $\text{Re}(x)>0$. It is striking to see that the previous identity is actually equivalent to the functional equation
\begin{equation}
(2\pi)^{-s}\,\Gamma(s)\,L(s) = (2\pi)^{-(12-s)}\,\Gamma (12-s)\,L(12-s),\,\,\,\,\,\text{with } L(s):=\sum _{n=1}^{\infty}\frac{\tau(n)}{n^s},\,\,\,\text{Re}(s)>\frac{13}{2}. \label{hecke_for_tau_intro}
\end{equation}
Also, Chandrasekharan and Narasimhan employed these general summation formulas to study the average order of the arithmetical functions $a(n)$ and $b(n)$, furnishing general analogues of the divisor and circle problems \cite{average_order_chandrasekharan}.\footnote{The fact that generalized summation formulas may be connected to generalized divisor problems should be no surprise, as one of the first applications that Vorono\"i found for his formula was the first nontrivial estimate for the divisor problem \cite{berndt_divisor}.} In a joint paper with S. Bochner \cite{arithmetical identities}, Chandrasekharan also studied uniqueness properties of Dirichlet series satisfying a certain class of functional equations, proving generalized versions of the celebrated Hamburger's Theorem. 

\bigskip{}

It is within this framework of functional equations and their equivalent modular transformations that the early work of Berndt must be understood.
In his early papers \cite{dirichletserisI, dirichletserisII, dirichletserisIII, dirichletserisIV, dirichletserisV, dirichletserisVI, dirichletserisVII}, Bruce C. Berndt followed
the main guidelines of the previous authors, while placing particular emphasis on concrete interesting cases, such as the ones appearing in the theory of Epstein zeta functions.\footnote{Before his deep engagement with Ramanujan’s work, Berndt devoted much of his early research to summation formulas in the spirit of Hecke, Bochner, Chandrasekharan, and Narasimhan. He was also keenly interested in the distribution of zeros of general Dirichlet series, as evidenced by \cite{berndt_zeros_derivative, berndt_zeros_(i), berndt_zeros_(ii)}. In this direction, he extended the Potter–Titchmarsh method \cite{Titchmarsh_Potter}, originally developed to prove that the Epstein zeta function possesses infinitely many critical zeros, to the case of the Dedekind zeta function \cite{berndt_dedekind}. According to Berndt’s own accounts, his paper \cite{berndt_modular_ramanujan}, together with its connection to contemporaneous results of Emil Grosswald \cite{grosswald_ramanujan}, was instrumental in drawing him toward the work of Srinivasa Ramanujan.} In his first paper \cite{dirichlet and hecke}, based on his doctoral dissertation, Berndt generalized the formula of Dixon and Ferrar (\ref{III-106}) by replacing the specific arithmetic function $r_{2}(n)$ with a general function $a(n)$, subject only to the condition that the associated Dirichlet series $\phi(s)$ satisfy a Hecke-type functional equation (\ref{Hecke type Functional Equations}). As a consequence of this general transformation formula, Berndt obtained, as a special case, a broad extension of the identity of Chandrasekharan and Narasimhan (\ref{chandra nara intro book}), namely
\begin{equation}
\sum_{n=1}^{\infty}\tau(n)\,n^{\frac{\nu+1}{2}}K_{\nu+1}\left(x\sqrt{n}\right)=2^{36+\nu}x^{\nu+1}\pi^{12}\Gamma\left(13+\nu\right)\,\sum_{n=1}^{\infty}\frac{\tau(n)}{\left(x^{2}+16\pi^{2}n\right)^{\nu+13}}, \label{berndt_for_tau_intro_book}
\end{equation}
where $\text{Re}(x),\,\text{Re}(\nu)>0$.\footnote{The reader can actually recover (\ref{chandra nara intro book}) by taking $\nu=-1/2$ on Berndt's general formula.} As the formula obtained by Chandrasekharan and Narasimhan, (\ref{chandra nara intro book}), Berndt’s Bessel expansion for $\tau(n)$ (\ref{berndt_for_tau_intro_book}) is not merely a consequence of Hecke's functional equation (\ref{hecke_for_tau_intro}) but is in fact equivalent to it.\footnote{As far as we know, the earliest publication of this equivalence is given in [\cite{dirichletserisIII}, p. 342, Theorem 8.1]. The reader may also check Theorem 1.1 of \cite{Kanemitsu_Modular_Ramanujan}, where three equivalent statements to Hecke's functional equation (\ref{Hecke type Functional Equations}), including the generalized version of (\ref{berndt_for_tau_intro_book}), are stated.}  

Berndt unified several aspects of this whole branch of Analytic Number
Theory, often connecting the general theory of Dirichlet series with
the \textit{Ramanujan Notebooks} Program, a powerful source of motivation that furnished many of the most interesting examples in the whole area of summation formulas \cite{Guinand_Ramanujan}. Berndt’s enthusiasm for Ramanujan’s mathematics proved contagious and profoundly influenced his students. Among those most deeply influenced was Atul Dixit, who, in joint work with Berndt and Sohn \cite{koshliakov_ramanujan_character}, proved that several transformation formulas discovered by Ramanujan admit natural extensions to Dirichlet $L$-functions.

While working on Ramanujan's lost notebook, Berndt and Dixit \cite{berndt_dixit_digamma} offered
two beautiful proofs of an outstanding formula stated on page 220 of Ramanujan's lost notebook. This formula states
that, if $\alpha,\beta$ are two positive real numbers such that $\alpha\beta=1$,
then the following summation formula holds
\begin{align}
&\sqrt{\alpha}\left\{ \frac{\gamma-\log(2\pi\alpha)}{2\alpha}+\sum_{n=1}^{\infty}\left(\psi\left(n\alpha\right)+\frac{1}{2n\alpha}-\log\left(n\alpha\right)\right)\right\} \nonumber \\
& =\sqrt{\beta}\left\{ \frac{\gamma-\log(2\pi\beta)}{2\beta}+\sum_{n=1}^{\infty}\left(\psi\left(n\beta\right)+\frac{1}{2n\beta}-\log\left(n\beta\right)\right)\right\} \nonumber \\
&=-\frac{1}{\pi^{\frac{3}{2}}}\,\intop_{0}^{\infty}  \left|\Xi\left(\frac{t}{2}\right)\Gamma\left(\frac{-1+it}{4}\right)\right|^{2}\,\frac{\cos\left(\frac{1}{2}t\log\alpha\right)}{1+t^{2}}dt,\label{Integral representation Ramanujan formula Dixit Berndt first paper}
\end{align}
where $\gamma$ denotes the Euler-Mascheroni constant, $\psi(z):=\Gamma^{\prime}(z)/\Gamma(z)$
corresponds to the digamma function and $\Xi(t)$ is the classical
Riemann $\Xi-$function, defined as
\[
\Xi\left(t\right)=\xi\left(\frac{1}{2}+it\right),\,\,\,\,\,\,\,\text{with }\xi(s)=\frac{1}{2}s(s-1)\,\pi^{-\frac{s}{2}}\Gamma\left(\frac{s}{2}\right)\zeta(s).
\]
Ramanujan’s ingenious idea to express (\ref{Integral representation Ramanujan formula Dixit Berndt first paper}) and related formulas through integrals involving the Riemann zeta function on the critical line proved decisive for Dixit’s doctoral work under Berndt's guidance. Dixit substantially extended Ramanujan’s method of deriving summation formulas, obtaining several remarkable generalizations of his results. Among these developments, we find particularly striking Dixit’s generalization of Hardy’s digamma identity [\cite{Dixit_theta}, p. 375]
\begin{align}
\sqrt{\alpha}\,e^{z^{2}/8}\intop_{0}^{\infty}\left(\psi(u+1)-\log u\right)\,e^{-\pi\alpha^{2}u^{2}}\cos\left(\sqrt{\pi}\alpha uz\right)\,du\nonumber \\
=\sqrt{\beta}\,e^{-z^{2}/8}\,\intop_{0}^{\infty}\left(\psi(u+1)-\log u\right)\,e^{-\pi\beta^{2}u^{2}}\cosh\left(\sqrt{\pi}\beta uz\right)\,du.\label{digamma hardy inttttro}
\end{align}
The reader may find several further results of a similar nature in \cite{dixit_RiemannXI, dixit_series}.

\bigskip{}

Up to this point, we have reviewed the historical development of the understanding of summation formulas and, at this stage, the reader may wonder what unifies this vast array of results - ranging from Poisson’s formula to Vorono\"i’s identities and Ramanujan’s transformations. The answer, as we shall emphasize throughout this paper, lies in a single analytic mechanism: the \textbf{Mellin transform}.
 
In an interesting survey, Berndt \cite{dirichletserisV} observes that the hypothesis used in the proofs of
Vorono\"i's formulas (\ref{Intro - 7}) and (\ref{Intro - 8}) generally fall into three broad classes. The first consists of smooth functions, typically $f\in C^{1}[a,b]$ or $f\in C^{2}[a,b]$.
The second comprises functions of bounded variation on $[a,b]$, a class considered in the finite versions of (\ref{Intro - 7}) and (\ref{Intro - 8}). 
Lastly, a third approach relies on the theory of functions in $L_{2}\left(\mathbb{R}_{+}\right)$ together with the techniques of Mellin and Hankel transforms. This third perspective, motivated by the $L_{2}$ analogues of the Fourier and Hankel transforms developed in Titchmarsh's classical treatise \cite{titchmarsh_fourier_integrals}, was initiated by Ferrar \cite{ferrar_I, ferrar_II} and Guinand \cite{guinand_poisson, guinand_I, guinand_II}.\footnote{Guinand \cite{guinand_poisson} also remarks that his $L_{2}$ analogue of the Poisson summation formula admits extensions to Dirichlet characters.} Pearson \cite{pearson} later extended the $L_{2}$ method to the summation formula involving $r_{2}(n)$, namely (\ref{Intro - 8}), while Nasim employed integral transforms within the $L_{2}$ framework to establish (\ref{Intro - 7}) and related formulas \cite{nasim_sigmak, nasim_sigma, nasim_Voronoi}. Despite his extensive contributions to the theory of Dirichlet series, Berndt remarks that he never attempted to develop a general theory of summation formulas within the $L_{2}$ framework.\footnote{To the best of our knowledge, a fully satisfactory set of general conditions ensuring the validity of $L_{2}$ summation formulas for arbitrary Dirichlet series has not yet been established.}
This provided a starting point for Semyon's work on summation formulas, initiated 15 years ago \cite{Yakubovich_Voronoi}. In the
papers \cite{Yakubovich_Voronoi, new_summation_formulas, Yakubovich_Nasim}, he revised the problem raised
by Vorono\"i by introducing two new classes of functions. The first consists of functions that are absolutely
continuous on $\mathbb{R}_{+}$ and whose Mellin transform, $f^{*}(s)$,
satisfies\footnote{while condition (\ref{Intro - 11}) is natural in the $L_{2}$ setting, the results in \cite{Yakubovich_Voronoi, new_summation_formulas, Yakubovich_Nasim} were established only for special Dirichlet series, such as $\zeta^k(s)$ and $\zeta(s-i\tau)\zeta(s+i\tau)$, $\tau >0$, and not for general Dirichlet series.}
\begin{equation}
\intop_{\frac{r}{2}-i\infty}^{\frac{r}{2}+i\infty}\left|s\,f^{*}(s)\right|^{2}|ds|<\infty,\,\,\,\,\,\text{Re}(s)=\frac{r}{2} \text{ is the critical line of $\phi(s)$.}\label{Intro - 11}
\end{equation}
As I have emphasized in my Master’s thesis \cite[pp. 134–135, 167–169]{thesis}, proving summation formulas within the class defined by (\ref{Intro - 11}) requires nontrivial analytic input. In particular, one needs deep estimates for the relevant arithmetical functions (divisor-type problems), as well as subconvex bounds for the Dirichlet series $\phi(s)$ and $\psi(s)$, (\ref{Dirichlet series}), on their critical lines $\Re(s)=r/2$.\footnote{By “subconvex” we mean any bound of smaller order than that predicted by the Phragm\'en–Lindel\"of principle. For instance, in order to establish (\ref{Intro - 7}) within the class (\ref{Intro - 11}), one must invoke the van der Corput estimate $\zeta(1/2+it)\ll_{\varepsilon} |t|^{1/6+\varepsilon}$. This illustrates the difficulty of formulating general conditions for arbitrary arithmetical functions, since subconvex bounds on the critical line cannot be guaranteed apriori for a general Dirichlet series.}

Since these conditions are, in general, quite difficult to ensure, Semyon introduced a second class of functions (the ''M\"untz Class'') in a later publication \cite{new_summation_formulas} to overcome certain restrictions within the condition (\ref{Intro - 11}).  
Essentially, for $n\geq2$, a function of M\"untz-type, $f\in \mathcal{M}_{\alpha, n}$, is a function belonging to $C^{n}\left(\mathbb{R}_{0}^{+}\right)$ which decays, as well as
its first $n$ derivatives, in the form $f^{(k)}(x)=O\left(x^{-\alpha-k}\right),$
for $\alpha>1$ and $x\rightarrow\infty$ \cite{thesis, new_summation_formulas}. He also studied summation formulas under this class involving the M\"obius
function $\mu(n)$, offering equivalent statements to the Riemann Hypothesis.

\bigskip{}

When I first read \cite{new_summation_formulas}, I was genuinely confused. The opening result of the paper was a proof of the Poisson summation formula, yet the argument relied on the functional equation of the Riemann zeta-function. My immediate reaction was: \textit{"why use a result whose proof depends on the Poisson summation formula in order to establish the Poisson summation formula itself?"}.
After several rereadings of the first statement in this paper, I realized that the situation was more subtle. The functional equation of the Riemann zeta-function ultimately depends only on a very specific and easier instance of the Poisson summation formula, namely, the theta transformation formula (\ref{transf formula intro theta book}). However, as discussed in our earlier review of Cauchy’s work, this only requires a version of the Poisson summation formula valid for analytic functions.
In this way, the apparent circularity is resolved: the functional equation is derived from a weaker form of the Poisson summation formula (the analytic version), and is then used to establish a significantly stronger version, valid for the M\"untz class.\footnote{Which is way more general than the class of analytic functions in $\mathbb{C}$! In fact, Semyon told me later that his main motivation was to develop a rigorous argument for a semi-formal proof given in Titchmarsh's text [\cite{titchmarsh_fourier_integrals}, pp. 60 - 64].}
 
\bigskip{}

The connection between the functional equation for the Riemann zeta-function and the Poisson summation was constructed by Semyon using Mellin transforms. As noted in [\cite{dirichletserisV}, p. 140] and already emphasized in the introduction, Ferrar appears to have been the first mathematician to clearly articulate this point of view: that the functional aspects of a summation formula are intimately tied to the analytic behavior of the underlying Dirichlet series via the Mellin transform.
In fact, in a relatively little-known paper \cite{ferrar},
Ferrar embarked on a study of summation formulas that possess
symmetric properties. To give an example, let us recall Koshliakov's formula
(\ref{III-84}), already introduced as the textbook first application of Vorono\"i's formula, 
\begin{equation}
\sqrt{x}\left\{ \gamma-\log\left(\frac{4\pi}{x}\right)+4\,\sum_{n=1}^{\infty}d(n)\,K_{0}\left(2\pi nx\right)\right\}=\frac{1}{\sqrt{x}}\left\{ \gamma-\log\left(4\pi x\right)+4\,\sum_{n=1}^{\infty}d(n)\,K_{0}\left(\frac{2\pi n}{x}\right)\right\} ,\label{Koshliakov informal intro RHUM}
\end{equation}
where $d(n):=\sum_{d|n}1$ is the classical divisor function and $x>0$.
Just like the theta transformation formula (\ref{transf formula intro theta book}), Koshliakov's formula is quite symmetrical. A formal way to view it is as 
just one example of a ``modular'' relation of the form
\begin{equation}
\mathcal{F}(x)=\mathcal{F}\left(\frac{1}{x}\right),\,\,\,\,\,x>0,\label{fgeneral relation ferrar setting}
\end{equation}
where $\mathcal{F}$ is a function of arithmetical interest.\footnote{In the case of (\ref{Koshliakov informal intro RHUM}), $\mathcal{F}(x)$ represents an infinite
series containing $d(n)$ and some residual terms.}

Ferrar proved (\ref{Koshliakov informal intro RHUM}) using the theory of Mellin transforms, an accessible way to make his general theory having very interesting consequences. 
Besides Koshliakov's formula (\ref{Koshliakov informal intro RHUM}),
Ferrar provided another interesting transformation of the form (\ref{fgeneral relation ferrar setting}), offering the following interesting example. 

\begin{theorem*}\label{Ferrar Theorem classic}
If $\alpha,\beta>0$ are such that $\alpha\beta=1$, then the following
transformation formula holds
\begin{align}
2\sqrt{\alpha}\,\sum_{n=1}^{\infty}\left\{ e^{\frac{\pi n^{2}\alpha^{2}}{2}}K_{0}\left(\frac{\pi n^{2}\alpha^{2}}{2}\right)-\frac{1}{\alpha n}\right\} +\frac{1}{\sqrt{\alpha}}\left(\gamma-\log(16\pi)-2\log\left(\alpha\right)\right)\nonumber \\
=2\sqrt{\beta}\,\sum_{n=1}^{\infty}\left\{ e^{\frac{\pi n^{2}\beta^{2}}{2}}K_{0}\left(\frac{\pi n^{2}\beta^{2}}{2}\right)-\frac{1}{\beta n}\right\} +\frac{1}{\sqrt{\beta}}\left(\gamma-\log(16\pi)-2\log\left(\beta\right)\right).\label{ferrar introduction section}
\end{align}
\end{theorem*}

Since the summation index ``$n^{2}$'' is suggestive of the one-dimensional
case of a single sum of squares, one may ask what is the generalization of
Ferrar's formula when the summation variable, ``$n^{2}$'', is replaced
by a sum of $k-$squares, that is, ``$n_{1}^{2}+...+n_{k}^{2}$'', and the resulting infinite series becomes a sequence of $k$ sums. 

Before presenting this generalization, we need to introduce a new special function that will help us finding it. The Whittaker function $W_{\mu,\nu}(x)$ is the solution of Whittaker's
differential equation {[}\cite{whittaker_watson}, p. 337{]},
\begin{equation*}
\frac{d^{2}u}{dx^{2}}+\left(\frac{\mu}{x}+\frac{1}{4x^{2}}-\frac{\nu^{2}}{4x^{2}}-\frac{1}{4}\right)u=0,
\end{equation*}
which is determined uniquely by the property\footnote{here, we suppose that
$\mu\in\mathbb{R}$.}
\begin{equation*}
W_{\mu,\nu}(x)\sim x^{\mu}e^{-\frac{x}{2}},\,\,\,\,x\rightarrow\infty.
\end{equation*}
The theory of Whittaker's functions is usually derived from the study of confluent hypergeometric functions. One can see from this theory that Whittaker's function has the asymptotic behavior
\[
W_{\mu,\nu}(x)=O\left(x^{\text{Re}(\nu)+\frac{1}{2}}\right)+O\left(x^{\frac{1}{2}-\text{Re}(\nu)}\right),\,\,\,\,\nu\neq0\,\,\,\,x\rightarrow0^{+},
\]
\[
W_{\mu,0}(x)=O\left(x^{\frac{1}{2}}\,\log(x)\right),\,\,\,\,x\rightarrow0^{+}
\]
and
\begin{equation}
W_{\mu,\nu}(x)=O\left(x^{\mu}\,e^{-x/2}\right),\,\,\,\,x\rightarrow\infty. \label{asymp behavior whiutttttakkkkkeeeer rhum}
\end{equation}
When we set its first index, $\mu$, as zero, the function $W_{\mu,\nu}(x)$ reduces
to the modified Bessel function of the second kind, $K_{\nu}(x)$, and, in a different arrangement of the indices, it can be
reduced to an elementary function {[}\cite{NIST}, p. 338, 13.18.2,
13.18.9{]}. For example, this means that we can retrieve the reduction formulas
\begin{equation}
W_{0,\nu}\left(2x\right)=\sqrt{\frac{2x}{\pi}}\,K_{\nu}(x),\label{Whittaker to MAcdonald RHUM}
\end{equation}
\begin{equation}
W_{\nu+\frac{1}{2},\nu}(x)=x^{\nu+\frac{1}{2}}e^{-\frac{x}{2}}\label{Whittaker to exponential RHUM}
\end{equation}
as particular examples. 
The first of these formulas, (\ref{Whittaker to MAcdonald
RHUM}), will be of special importance in this paper, as we hope to show in Corollary \ref{corollary as ferrar once afgain} in the third section below.  

One of the main objectives of our contribution to this special volume is to prove a generalization of Ferrar's formula (\ref{ferrar introduction section}), in a way that involves a generalization of the modified Bessel function, $K_{0}(x)$, but also an extension of (\ref{ferrar introduction section}) to the arithmetical function $r_{k}(n)$, which counts the number of representations of a given $n$ as a sum of $k$ squared integers. 

\begin{theorem}\label{gen Ferrar paper Petro}
Let $r_{k}(n)$ represent the number of ways of representing $n$ as a sum of $k$ squares and let $\alpha,\beta>0$ be such that $\alpha \beta=1$. Moreover, for a vector $\mathbf{x}\in\mathbb{R}^{n}$,
let $|\mathbf{x}|$ denote its Euclidean norm.

Then the following summation formula holds
\begin{align}
\alpha^{\frac{k}{2}-1}\,\sum_{n=1}^{\infty}r_{k}(n)\left\{ \frac{e^{\frac{\pi n\alpha^{2}}{2}}}{\sqrt{n}}W_{\frac{1-k}{2},0}\left(\pi n\alpha^{2}\right)-\frac{\pi^{\frac{1-k}{2}}}{n^{\frac{k}{2}}\alpha^{k-1}}\right\} +\pi^{\frac{1-k}{2}}\,\alpha^{-\frac{k}{2}}\zeta_{k-1}\left(\frac{k}{2}\right)\nonumber \\
+\frac{\sqrt{\pi}}{\Gamma\left(\frac{k}{2}\right)\alpha^{\frac{k}{2}}}\left(2\gamma+\psi\left(\frac{k}{2}\right)-2\log\left\{ \prod_{\mathbf{m}\in\mathbb{Z}^{k-1}\setminus\{\mathbf{0}\}}\left(1-e^{-2\pi|\mathbf{m}|}\right)\right\} -\log\left(4\pi\alpha^{2}\right)\right)\nonumber \\
=\beta^{\frac{k}{2}-1}\,\sum_{n=1}^{\infty}r_{k}(n)\left\{ \frac{e^{\frac{\pi n\beta^{2}}{2}}}{\sqrt{n}}W_{\frac{1-k}{2},0}\left(\pi n\beta^{2}\right)-\frac{\pi^{\frac{1-k}{2}}}{n^{\frac{k}{2}}\beta^{k-1}}\right\} +\pi^{\frac{1-k}{2}}\beta^{-\frac{k}{2}}\zeta_{k-1}\left(\frac{k}{2}\right)\nonumber \\
+\frac{\sqrt{\pi}}{\Gamma\left(\frac{k}{2}\right)\beta^{\frac{k}{2}}}\left(2\gamma+\psi\left(\frac{k}{2}\right)-2\log\left\{ \prod_{\mathbf{m}\in\mathbb{Z}^{k-1}\setminus\{\mathbf{0}\}}\left(1-e^{-2\pi|\mathbf{m}|}\right)\right\} -\log\left(4\pi\beta^{2}\right)\right), \label{Formula GENERAL Ferrar!}
\end{align}
where $\gamma$ is the Euler-Mascheroni constant, $\psi(z)$ is Euler's
digamma function and $W_{\mu,\nu}(z)$ denotes the Whittaker function of the second kind.
\end{theorem}

The previous theorem is a generalization of (\ref{ferrar introduction section}) due to the very nature of the function $r_{k}(n)$. Since $r_{k}(n):= \# \left\{(n_{1},..., n_{k})\, :\, n=n_{1}^2+...+n_{k}^2 \right\}$, one can easily see that $r_{1}(n)=2$ if $n$ is a perfect square and zero otherwise. Replacing $k=1$ in (\ref{ferrar introduction section}), one can see that the infinite product over $\mathbb{Z}^{k-1} \setminus \{\mathbf{0}\}$ does not make sense, so it contributes as a zero term and the resulting summation formula will be precisely (\ref{ferrar introduction section})!
Just like Ferrar, we can also obtain a generalization not only of (\ref{ferrar introduction section}) but also of Koshliakov's formula (\ref{Koshliakov informal intro RHUM}) involving the general arithmetical function $r_{k}(n)$. The second result that we present in this paper is given as follows. 

\begin{theorem}\label{Koshliakov theorem rk(n)}

Let $\alpha,\beta$ be two positive real numbers such that $\alpha\beta=\pi^{2}$.
Then the following identity holds
\begin{align}
\alpha^{\frac{k}{2}}\left\{ 2\eta_{k-1}\left(\frac{k}{2}\right)+2\gamma-4\log\left\{ \prod_{\mathbf{m}\in\mathbb{Z}^{k-1}\setminus\{\mathbf{0}\}}\left(1-e^{-2\pi|\mathbf{m}|}\right)\right\}-2\log\left(4\beta\right)+2\sum_{m,n=1}^{\infty}r_{k}(m)\,r_{k}(n)\,K_{0}\left(2\sqrt{m\,n}\,\alpha\right)\right\}\nonumber \\
=\beta^{\frac{k}{2}}\left\{ 2 \eta_{k-1}\left(\frac{k}{2}\right)+2\gamma-4\log\left\{ \prod_{\mathbf{m}\in\mathbb{Z}^{k-1}\setminus\{\mathbf{0}\}}\left(1-e^{-2\pi|\mathbf{m}|}\right)\right\}-2\log\left(4\alpha\right)+2\sum_{m,n=1}^{\infty}r_{k}(m)\,r_{k}(n)\,K_{0}\left(2\sqrt{m\,n}\beta\right)\right\}. \label{general_KOSH_rk(n)}
\end{align}

\end{theorem}

Why is (\ref{general_KOSH_rk(n)}) a generalization of Koshliakov's formula (\ref{Koshliakov informal intro RHUM})?  To see why this is the case, let $\tilde{d}_{k}(n)$ denote the generalized divisor function given by $\tilde{d}_{k}(n):=\sum_{d|n}r_{k}(d)\,r_{k}\left(\frac{n}{d}\right)$.
Then (\ref{general_KOSH_rk(n)}) is equivalent to the transformation formula
\begin{align*}
\alpha^{\frac{k}{2}}\left\{ 2\pi^{-\frac{k}{2}}\Gamma\left(\frac{k}{2}\right)\zeta_{k-1}\left(\frac{k}{2}\right)+2\gamma-4\log\left\{ \prod_{\mathbf{m}\in\mathbb{Z}^{k-1}\setminus\{\mathbf{0}\}}\left(1-e^{-2\pi|\mathbf{m}|}\right)\right\}-2\log\left(4\beta\right)+2\sum_{n=1}^{\infty}\tilde{d}_{k}(n)\,K_{0}\left(2\sqrt{n}\,\alpha\right)\right\} \\
=\beta^{\frac{k}{2}}\left\{ 2 \eta_{k-1}\left(\frac{k}{2}\right)+2\gamma-4\log\left\{ \prod_{\mathbf{m}\in\mathbb{Z}^{k-1}\setminus\{\mathbf{0}\}}\left(1-e^{-2\pi|\mathbf{m}|}\right)\right\}-2\log\left(4\alpha\right)+2\sum_{n=1}^{\infty}\tilde{d}_{k}(n)\,K_{0}\left(2\sqrt{n}\,\beta\right)\right\}. 
\end{align*}
Since $r_{1}(n)=2$ iff $n$ is a perfect square and 0 otherwise, it is quite simple to check that $\tilde{d}_{1}(n) = 4\,d(n)$ if $n$ is a perfect square, while $\tilde{d}_{1}(n)=0$ otherwise. Therefore, (\ref{general_KOSH_rk(n)}) reduces to (\ref{Koshliakov informal intro RHUM}) in the case where $k=1$.  

\bigskip{}

So why have I chosen to focus on a paper of Ferrar in this tribute to Semyon? The answer lies in a fundamental shift of perspective that Ferrar was among the first to articulate with full clarity. He recognized that summation formulas are not merely isolated analytical identities, but rather manifestations of deeper structural relations: the functional equations satisfied by the associated Dirichlet series.

Before Ferrar, the prevailing approach was highly technical: one established summation formulas through intricate analytic arguments and only afterward extracted their consequences. Ferrar himself initially followed this path in his joint work with Dixon \cite{dixon_ferrar_lattice, dixon_ferrar_circle(i), dixon_ferrar_circle(ii)}. However, he later realized that these formulas were not the starting point, but rather the shadow cast by a more fundamental principle. In his subsequent work \cite{ferrar}, he made this philosophy explicit by showing that summation formulas and functional equations are, in essence, equivalent formulations of the same underlying phenomenon.

At the heart of this realization lies the \textbf{Mellin transform}, which serves as a natural bridge between Dirichlet series and integral transforms. It is precisely this point of view that has profoundly influenced Semyon’s work, ever since he started working under O. Marichev at the age of 19, at a time when he was simultaneously engaged in both academic study and professional work.\footnote{At the time, he balanced his studies with a demanding technical job that often required night shifts. This early display of discipline and commitment is perhaps indicative of the remarkable work ethic that still defines his personal and professional life: his days begin as early as 5 a.m., followed by a regular morning run of approximately 4 km in Oliveira do Douro, before heading to the Faculty, often arriving so early that, in the absence of any staff at that hour, he is the one who opens the doors for the baker delivering supplies to the caf\'e.}
Much like Ferrar before him, he early understood that the Mellin transform is not simply a technical tool, but a unifying framework through which one may systematically construct and invert a wide class of integral transforms, and thereby generate new summation formulas.\footnote{Semyon's interest in summation formulas started after long months of work on Salem's problem. As far as I could see, it was a way for him to disconnect from the heavier theory of integral transforms.}
This way of dealing with integral transforms and, in a broader sense, with special functions themselves, is perhaps best captured in his own words, expressed repeatedly in our conversations:
\begin{center}
\textit{“Marichev really had an ingenious idea. That he could use a universal transform, the Mellin transform, and, from it, study all the others.”}\,\,\,\footnote{Semyon Yakubovich, circa Aug 2022. We were discussing the Fourier
transform of the Whittaker function with respect to its second index. This discussion resulted in the joint paper \cite{RYPIT}.}
\end{center}
\begin{center}
\textit{“I do not trust the formulas contained in the usual books (G.R., P.M.B., etc.). I essentially use the table of Mellin transforms created by Marichev.”}\,\,\,\footnote{Semyon Yakubovich, in several of our discussions. The table he was referring to is Marichev's handbook \cite{marichev_83}, which has been recently replaced by the more complete table \cite{handbook_marichev}.}
\end{center}

These remarks reflect a methodological stance: rather than relying on isolated identities, one works within Mellin transform's space in which such identities emerge naturally from functional relations. It was precisely this perspective that guided Semyon’s research from an early stage, when he simplified Wimp's formula as a Master's student, and which continues to shape his contributions to the theory of integral transforms and summation formulas.

In this context, the role of the Mellin transform becomes not merely technical, but conceptual. It provides the language through which functional equations, modular-type transformations, and summation formulas may be seen as different expressions of a single analytic structure. This unifying power was remarkably highlighted by Bochner \cite{bochner_fourier}, who wrote:
\begin{center}
\textit{“Gauss had to work hard to derive his reciprocity law for Gaussian sums. However, Dirichlet did it smoothly after him with the Poisson summation formula for Fourier series (...). Erich Hecke, Dirichlet’s counterpart in the twentieth century (...) bested eminent Fourier analysts of his generation by \textbf{operating intricately with the Mellin transform} on the body of the Poisson summation formula in several variables.”}
\end{center}

In my own case, I can only say that I have learned a great deal from him during our fruitful collaboration that lasted from 2020 to 2024. The fact that I now instinctively view summation formulas through the action of their associated Mellin transforms is, perhaps, a small indication that I have absorbed at least an infinitesimal portion of Semyon's deep understanding of classical analysis.

In this spirit, all the proofs that follow are carried out within the framework of the Mellin transform, highlighting its role as a decisive tool in the theory of summation formulas. From this perspective, intricate expressions involving Bessel and Whittaker functions are reduced to simple ratios of gamma functions, which may be handled with a fluency not unlike that of elementary algebra.

\bigskip{}

Ralph Waldo Emerson once wrote that \textit{“The essence of greatness is the perception that virtue is enough”}. In my experience, Semyon embodies this principle in a way that is both rare and instructive. In an environment where recognition, position, visibility and, above all, unrestrained vanity can too easily become ends in themselves, he has consistently remained guided by a quieter standard: that of intellectual honesty, personal discipline, and continuous self-improvement.

This disposition became particularly evident to me at a decisive moment in my own journey. At the conclusion of my PhD, I chose to pursue a career outside academia, a decision that, in my immediate surroundings, was often met with surprise and, more often than not, with incomprehension. Semyon, however, was unwavering in his support. He understood, with the same clarity one associates with mathematical truth, the structural difficulties faced by early-career researchers, and recognized that such a path would not, in my case, offer a sustainable or dignified prospect.\footnote{It may shock some of the readers but, at present, a typical postdoctoral position in Portugal corresponds to a net monthly income of approximately 1901 euros, paid over 12 months (and not 14). This places it only marginally above the national average salary, despite the level of specialization required. Moreover, such positions generally do not include standard labor protections, social benefits, or long-term stability, most of them being short-term contracts. By contrast, entry-level positions for PhD holders in the private sector often offer significantly higher compensation, frequently exceeding this value by a margin of at least 170\%, together with additional benefits.}

In this, he stands in marked contrast with a more common academic posture, in which younger researchers are encouraged to remain within the system under the implicit promise of future stability, often at the cost of prolonged uncertainty. In less favorable instances, which occur far more frequently than is acknowledged, this dynamic fosters a form of \textbf{structural dependence}, whereby early-career researchers remain in extended states of precarity while sustaining the productivity of those above them within a fundamentally pyramid-like scheme. 
This structure is particularly troubling and self-perpetuating: the very conditions of precarity that constrain younger researchers simultaneously contribute to the maintenance and, at times, the quiet reinforcement of the system itself, whose upper layers it disproportionately benefits and which cannot, in practice, absorb those at its base.

His advice, by contrast, was guided not by institutional interest, but by a genuine concern for my well-being and future, long-term stability, and the ability to build a family. Having himself experienced a path marked by geographical instability, moving from Belarus through Japan, the Netherlands, and Belgium before eventually settling in Portugal, he knew from lived experience the disproportionate demands imposed by an academic career.

For that honesty, I remain deeply grateful.

\section{Preliminary Results}

The main results in this paper are generalizations of two known summation formulas that involve the arithmetical function $r_{k}(n)$. Therefore, we should understand first the Dirichlet series that are involved in this story. 
Indeed, if we construct the Dirichlet series attached to $r_{k}(n)$, 
\begin{equation}
\zeta_{k}(s)=\sum_{n=1}^{\infty}\frac{r_{k}(n)}{n^{s}},\,\,\,\,\text{Re}(s)>\frac{k}{2},\label{definition zeta k}
\end{equation}
then $\zeta_{k}(s)$ can be continued to the complex plane as a meromorphic
function possessing only a simple pole at $s=\frac{k}{2}$ with residue
$\pi^{k/2}/\Gamma(k/2)$. Moreover, it satisfies the functional equation
\begin{equation}
\pi^{-s}\Gamma\left(s\right)\zeta_{k}(s)=\pi^{s-\frac{k}{2}}\Gamma\left(\frac{k}{2}-s\right)\zeta_{k}\left(\frac{k}{2}-s\right).\label{functional equation zetak}
\end{equation}
As noted in the Introduction of this paper, $r_{1}(n)=2$ if $n$ is a perfect
square and zero otherwise. Therefore, (\ref{definition zeta k}) reduces
to the Riemann $\zeta-$function, i.e., 
\begin{equation}
\zeta_{1}(s):=\sum_{n=1}^{\infty}\frac{r_{1}(n)}{n^{s}}=2\sum_{n=1}^{\infty}\frac{1}{n^{2s}}=2\zeta(2s),\,\,\,\,\text{Re}(s)>\frac{1}{2}.\label{zeta 1(s) definition}
\end{equation}
Furthermore, (\ref{functional equation zetak}) with $k=1$ gives the functional equation
for Riemann's $\zeta-$function
\begin{equation}
\pi^{-s}\Gamma\left(s\right)\zeta(2s)=\pi^{s-\frac{1}{2}}\Gamma\left(\frac{1}{2}-s\right)\zeta\left(1-2s\right).\label{Functional equation Riemann Popov}
\end{equation}
Like the functional equation for the Riemann zeta-function (\ref{Functional equation Riemann Popov}), there is a very symmetrical way to write (\ref{functional equation zetak}) as 
\begin{equation}
\eta_{k}(s) = \eta_{k}\left(\frac{k}{2}-s\right),\,\,\,\,\text{with }\,\eta_{k}(s):=\pi^{-s}\,\Gamma(s)\,\zeta_{k}(s). \label{def eta_k}
\end{equation}

In several occasions throughout this paper, 
we shall need to estimate the asymptotic order of $\zeta_{k}(s)$ resulting from the functional equation (\ref{functional equation zetak}). To justify most of the steps, we will often
invoke the following version of Stirling's formula 
\begin{equation}
\Gamma(\sigma+it)=(2\pi)^{\frac{1}{2}}\,t^{\sigma+it-\frac{1}{2}}\,e^{-\frac{\pi t}{2}-it+\frac{i\pi}{2}(\sigma-\frac{1}{2})}\left(1+\frac{1}{12(\sigma+it)}+O\left(\frac{1}{t^{2}}\right)\right),\label{Stirling exact form on Introduction}
\end{equation}
as $t\rightarrow\infty$, uniformly for $-\infty<\sigma_{1}\leq\sigma\leq\sigma_{2}<\infty$.
A similar formula can be written for $t$ tending to $-\infty$
by using the fact that $\Gamma(\overline{s})=\overline{\Gamma(s)}$. 

Like $\zeta(s)$, the Dirichlet series $\zeta_{k}(s)$ admits a Laurent expansion containing terms with arithmetical interest. 
The most important result of this section consists in finding the meromorphic expansion for $\zeta_{k}(s)$. Our proof will use a beautiful identity due to Popov [\cite{popov_formula_bessel}, eq. (6)] (cf. [\cite{sum_of_squares_berndt},
p. 329, Corollary 4.6.])
%
\begin{equation}
\frac{\beta ^{\nu /2}\Gamma \left (\nu +\frac{k}{2}\right )}{2\pi ^{\nu +\frac{k}{2}}}
\,\sum _{n=0}^{\infty}\frac{r_{k}(n)}{(n+\beta )^{\nu +\frac{k}{2}}}=\frac{\Gamma(\nu)}{2\pi^{\nu}\beta^{\nu/2}}+
\sum _{n=1}^{\infty}r_{k}(n)\,n^{\nu /2}K_{\nu}\left (2\pi \sqrt{n
\beta}\right ),
\label{ferrar_berndt_sum_of_squares_VTEX1}
\end{equation}
which is valid for any positive integer $k\geq 1$ and
$\text{Re}(\sqrt{\beta}),\,\text{Re}(\nu )>0$. The case $k=1$ of the previous formula was established by Watson
\cite{watson_reciprocal}, while studying self-reciprocal functions under the Hankel transform. This elegant identity of Watson, which formed a cornerstone of the techniques developed throughout my PhD journey, states that 
\begin{equation}
\sum_{n\in\mathbb{Z}}\frac{1}{\left(n^{2}+x^{2}\right)^{\nu}}=\frac{\sqrt{\pi}\,x^{1-2\nu}}{\Gamma(\nu)}\,\Gamma\left(\nu-\frac{1}{2}\right)+\frac{4\pi^{\nu}x^{\frac{1}{2}-\nu}}{\Gamma(\nu)}\,\sum_{n=1}^{\infty}n^{\nu-\frac{1}{2}}K_{\nu-\frac{1}{2}}\left(2\pi nx\right), \label{watson_intro_book}
\end{equation}
for any $x>0$ and $\text{Re}(\nu)>\frac{1}{2}$. Ferrar clearly knew this formula of Watson, as he and Dixon [\cite{dixon_ferrar_circle(i)}, p. 51] mimic one of the arguments given in \cite{watson_reciprocal} to study the analytical continuation of an infinite series involving the arithmetical function $r_{2}(n)$. Moreover, as the identity (\ref{III-106}) already suggests, Dixon and Ferrar also knew a particular case of (\ref{ferrar_berndt_sum_of_squares_VTEX1}) for $k=2$. 

We now state the generalization of Kronecker's
limit formula due to Epstein [\cite{epstein_I},
p. 644] as the next lemma. Although a proof of the forthcoming expansion is
already provided in Terras' paper [\cite{terras_epstein}, p. 485] (cf. [\cite{Siegel_Analytic Number Theory}, p. 22]), our notation and method
are slightly different and we rely on Popov's formula (\ref{ferrar_berndt_sum_of_squares_VTEX1}). Before proceeding, however, an important remark is in order. Throughout the introduction, we have emphasized the central role of the Mellin transform. It is therefore natural to ask whether the identity (\ref{ferrar_berndt_sum_of_squares_VTEX1}), upon which our argument depends, may itself be derived within this framework.
In the classical literature, formula (\ref{ferrar_berndt_sum_of_squares_VTEX1}) is typically obtained via the Vorono\"i summation formula, or alternatively as a consequence of the theta transformation formula (\ref{transf formula intro theta book}). However, as is already implicit in the work of Berndt \cite{dirichlet and hecke, dirichletserisIII}, a complete derivation may also be achieved through the use of Mellin–Barnes integral representations for the Bessel function $K_{\nu}(x)$.
This observation is fully consistent with the guiding principle adopted in this paper: that summation formulas, together with their associated transformations, may be systematically understood through the analytic machinery of the Mellin transform.

\begin{lemma}\label{continuation Epstein a la Selbegr chowla}
Let $k\in\mathbb{N}$ and $\zeta_{k}(s)$ be the Dirichlet series
attached to the sum of $k$ squares. Moreover, for a vector $\mathbf{x}\in\mathbb{R}^{n}$,
let $|\mathbf{x}|$ denote its Euclidean norm. Then $\zeta_{k}(s)$
admits the following Laurent expansion around its simple pole $s=\frac{k}{2}$,
\begin{align}
\zeta_{k}(s) & =\frac{\pi^{\frac{k}{2}}}{\Gamma\left(\frac{k}{2}\right)}\,\frac{1}{s-\frac{k}{2}}+\frac{\pi^{\frac{k}{2}}}{\Gamma\left(\frac{k}{2}\right)}\left(\gamma-2\log(2)-\psi\left(\frac{k}{2}\right)-2\log\left\{ \prod_{\mathbf{m}\in\mathbb{Z}^{k-1}\setminus\{\mathbf{0}\}}\left(1-e^{-2\pi|\mathbf{m}|}\right)\right\} \right)\nonumber \\
 & +\zeta_{k-1}\left(\frac{k}{2}\right) +O\left(s-\frac{k}{2}\right),\label{Laurent Expansion zeta k at s=00003Dk/2}
\end{align}
where $\psi(z)$ denotes Euler's digamma function, 
\[
\psi(z):=\frac{\Gamma^{\prime}(z)}{\Gamma(z)}.
\]
\end{lemma}

\begin{proof}[Proof of Lemma \ref{continuation Epstein a la Selbegr chowla}]
Let us note that, for $k\geq 2$ and $\text{Re}(s)>\frac{k}{2}$, $\zeta_{k}(s)$ may be represented in the following form
\begin{align}
\zeta_{k}(s) & =\sum_{n_{1},...,n_{k}\neq0}\,\frac{1}{\left(n_{1}^{2}+...+n_{k-1}^{2}+n_{k}^{2}\right)^{s}}=2\sum_{m,n\neq 0}^{\infty}\frac{r_{k-1}(m)}{(m+n^2)^{s}}\nonumber\\
 & =2\,\zeta_{k-1}(s)+2\,\sum_{n=1}^{\infty}\sum_{m=0}^{\infty}\frac{r_{k-1}(m)}{\left(m+n^{2}\right)^{s}}. \label{zetak ad Z2 RHUM}
\end{align}
In order to treat the infinite series with respect to the index $m$, we can use Popov's formula (\ref{ferrar_berndt_sum_of_squares_VTEX1}) with $\beta=m^2$: applying it gives the following expression for $\zeta_{k}(s)$, 
\begin{align}
\zeta_{k}(s) & =\zeta_{k-1}(s)+\frac{2\pi^{\frac{k-1}{2}}}{\Gamma(s)}\,\Gamma\left(s-\frac{k-1}{2}\right)\zeta\left(2s-k+1\right)\nonumber \\
 & +\frac{4\pi^{s}}{\Gamma(s)}\,\sum_{m,n=1}^{\infty}r_{k-1}(m)\left(\frac{m}{n^{2}}\right)^{\frac{s}{2}-\frac{k-1}{4}}\,K_{s-\frac{k-1}{2}}\left(2\pi\sqrt{m}\,n\right).\label{Application Selberg Chowla to zetak}
\end{align}
Let us study the right-hand side of this representation around the
point $s=\frac{k}{2}$. Recalling the definition of the digamma function, $\psi(z):=\Gamma^{\prime}(z)/\Gamma(z)$, 
and using the well-known Laurent expansions,
\[
\frac{1}{\Gamma(s)}=\frac{1}{\Gamma\left(\frac{k}{2}\right)}\left(1-\psi\left(\frac{k}{2}\right)\left(s-\frac{k}{2}\right)+O\left(\left(s-\frac{k}{2}\right)^{2}\right)\right),
\]
\[
\Gamma\left(s-\frac{k-1}{2}\right)=\sqrt{\pi}\left(1-\left(2\log(2)+\gamma\right)\left(s-\frac{k}{2}\right)+O\left(\left(s-\frac{k}{2}\right)^{2}\right)\right),
\]
\[
\zeta\left(2s-k+1\right)=\frac{1}{2s-k}+\gamma+O\left(s-\frac{k}{2}\right),
\]
we have 
\begin{align}
\frac{2\pi^{\frac{k-1}{2}}}{\Gamma(s)}\,\Gamma\left(s-\frac{k-1}{2}\right)\zeta\left(2s-k+1\right)\nonumber \\
=\frac{\pi^{\frac{k}{2}}}{\Gamma\left(\frac{k}{2}\right)}\,\frac{1}{s-\frac{k}{2}}+\frac{\pi^{\frac{k}{2}}}{\Gamma\left(\frac{k}{2}\right)}\left(\gamma-2\log(2)-\psi\left(\frac{k}{2}\right)\right) & +O\left(s-\frac{k}{2}\right).\label{Development first residual term selberg chowla in RHUM}
\end{align}
Since $K_{1/2}(x)=\sqrt{\frac{\pi}{2x}}\,e^{-x}$, the series involving
the Modified Bessel function in (\ref{Application Selberg Chowla to zetak}) has the following value at
$s=\frac{k}{2}$,
\begin{align*}
\frac{4\pi^{\frac{k}{2}}}{\Gamma\left(\frac{k}{2}\right)}\,\sum_{m,n=1}^{\infty}r_{k-1}(m)\left(\frac{m}{n^{2}}\right)^{\frac{1}{4}}K_{\frac{1}{2}}\left(2\pi\sqrt{m}\,n\right) & =\frac{2\pi^{\frac{k}{2}}}{\Gamma\left(\frac{k}{2}\right)}\,\sum_{m,n=1}^{\infty}\frac{r_{k-1}(m)\,e^{-2\pi\sqrt{m}\,n}}{n}\\
=\frac{2\pi^{\frac{k}{2}}}{\Gamma\left(\frac{k}{2}\right)}\,\sum_{m=1}^{\infty}r_{k-1}(m)\,\sum_{n=1}^{\infty}\frac{e^{-2\pi\sqrt{m}n}}{n} & =-\frac{2\pi^{\frac{k}{2}}}{\Gamma\left(\frac{k}{2}\right)}\,\sum_{m=1}^{\infty}r_{k-1}(m)\,\log\left(1-e^{-2\pi\sqrt{m}}\right).
\end{align*}
Using the definition of $r_{k-1}(m)$ and writing the last series
as multiple sum over $k-1$ variables of summation, we arrive at the expression
\begin{align}
&-\frac{2\pi^{\frac{k}{2}}}{\Gamma\left(\frac{k}{2}\right)}\,\sum_{m=1}^{\infty}r_{k-1}(m)\,\log\left(1-e^{-2\pi\sqrt{m}}\right) =-\frac{2\pi^{\frac{k}{2}}}{\Gamma\left(\frac{k}{2}\right)}\,\sum_{m_{1},...,m_{k-1}\neq0}\log\left(1-e^{-2\pi\sqrt{m_{1}^{2}+...+m_{k-1}^{2}}}\right)\nonumber \\
=&-\frac{2\pi^{\frac{k}{2}}}{\Gamma\left(\frac{k}{2}\right)}\,\log\left\{ \prod_{m_{1},...,m_{k}\neq0}\left(1-e^{-2\pi\sqrt{m_{1}^{2}+...+m_{k-1}^{2}}}\right)\right\} =-\frac{2\pi^{\frac{k}{2}}}{\Gamma\left(\frac{k}{2}\right)}\,\log\left\{ \prod_{\mathbf{m}\in\mathbb{Z}^{k-1}\setminus\{\mathbf{0}\}}\left(1-e^{-2\pi|\mathbf{m}|}\right)\right\} ,\label{Final expression dedekind eta almost at the end in fact}
\end{align}
where $\mathbf{x}\in\mathbb{R}^{n}$ and $|\mathbf{x}|$ denotes the usual
Euclidean norm of the vector $\mathbf{x}$. Combining (\ref{Final expression dedekind eta almost at the end in fact})
with (\ref{Development first residual term selberg chowla in RHUM})
and returning to the Selberg-Chowla formula (\ref{Application Selberg Chowla to zetak}),
we see that $\zeta_{k}(s)$ admits the Laurent expansion, 
\begin{align*}
\zeta_{k}(s) & =\frac{\pi^{\frac{k}{2}}}{\Gamma\left(\frac{k}{2}\right)}\,\frac{1}{s-\frac{k}{2}}+\frac{\pi^{\frac{k}{2}}}{\Gamma\left(\frac{k}{2}\right)}\left(\gamma-2\log(2)-\psi\left(\frac{k}{2}\right)-2\log\left\{ \prod_{\mathbf{m}\in\mathbb{Z}^{k-1}\setminus\{\mathbf{0}\}}\left(1-e^{-2\pi|\mathbf{m}|}\right)\right\} \right)\\
 &+\zeta_{k-1}\left(\frac{k}{2}\right) +O\left(s-\frac{k}{2}\right),
\end{align*}
which is exactly (\ref{Laurent Expansion zeta k at s=00003Dk/2}).
\end{proof}


\begin{remark}\label{analogues infinite product lemma}
The infinite product
\begin{equation}
\Phi_{k}(y):=\prod_{\mathbf{m}\in\mathbb{Z}^{k-1}\setminus\{\mathbf{0}\}}\left(1-e^{-2\pi|\mathbf{m}|y}\right),\,\,\,\,y>0,\label{generalization Dedekind}
\end{equation}
can be considered as an ``analogue'' of the square of Dedekind's
$\eta-$function, 
\begin{equation}
\eta(\tau):=e^{\frac{\pi i\tau}{12}}\,\prod_{m=1}^{\infty}\left(1-e^{2\pi im\tau}\right),\,\,\,\,\text{Im}(\tau)>0,\label{dedekind definition" RHUUUUM}
\end{equation}
when restricted to the imaginary axis $\tau=iy,$ $y>0$. To see why,
let us note that, when $k=2$, $\Phi_{k}(y)$ reduces to
\begin{equation}
\Phi_{2}(y):=\prod_{m\in\mathbb{Z}\setminus\{0\}}\left(1-e^{-2\pi|m|y}\right)=\left\{ \prod_{m=1}^{\infty}\left(1-e^{-2\pi my}\right)\right\} ^{2}:=e^{\frac{\pi y}{6}}\,\eta\left(iy\right)^{2}.\label{Phi 2 and connection to eta Dedekind}
\end{equation}
Therefore, as expected, (\ref{Laurent Expansion zeta k at s=00003Dk/2})
must imply a particular case of Kronecker's limit formula, whose classical
form arises when $k=2$. Indeed, using the fact that $\psi(1)=-\gamma$,
we have from (\ref{Laurent Expansion zeta k at s=00003Dk/2}) and
(\ref{Phi 2 and connection to eta Dedekind}) that 
\begin{equation*}
\zeta_{2}(s) =\frac{\pi}{s-1}+\zeta_{1}\left(1\right)+\pi\left(2\gamma-2\log(2)-4\log\left(\eta(i)\right)-\frac{\pi}{3}\right) +O\left(s-1\right).
\end{equation*}
But it is also well-known (cf. (\ref{Basel Intro}) above) that $\zeta_{1}(1):=2\zeta(2)=\frac{\pi^{2}}{3}$
and so, after elementary simplifications, we get the Laurent expansion
\begin{equation}
\zeta_{2}(s)=\frac{\pi}{s-1}+\pi\left(2\gamma-2\log(2)-4\log\left(\eta(i)\right)\right)+O(s-1),\label{Kronecker's limit formula for the partiuclar case diagonal quadratic}
\end{equation}
which is none other than Kronecker's famous result. 
\end{remark}

\begin{remark}\label{trivial remark when k=1}
When $k=1$, we know that neither $\zeta_{k-1}(s)$ nor $\Phi_{k}(1):=\prod_{\mathbf{m}\in\mathbb{Z}^{k-1}\setminus\{\mathbf{0}\}}\left(1-e^{-2\pi|\mathbf{m}|}\right)$
make much sense (besides, our proof below uses a formula, (\ref{zetak ad Z2 RHUM}), that only makes sense for $k\geq 2$). Therefore, in order to (\ref{Laurent Expansion zeta k at s=00003Dk/2}) be valid for any $k\in \mathbb{N}$, we take the convention
\begin{equation}
\zeta_{0}(s):=0,\,\,\,\,\Phi_{1}(y):=1. \label{silly convention but important}
\end{equation}
Recalling that $\zeta_{1}(s)=2\zeta(2s)$ and $\psi\left(\frac{1}{2}\right)=-2\log(2)-\gamma$,
a particular case of (\ref{Laurent Expansion zeta k at s=00003Dk/2}) under the convention (\ref{silly convention but important})
is
\[
2\zeta(2s)=\frac{1}{s-\frac{1}{2}}+2\gamma+O\left(s-\frac{1}{2}\right),
\]
which is the well-known Laurent expansion for Riemann's zeta function.
\end{remark}

\begin{remark}
There are some cases where the meromorphic expansion can be rewriten
in simpler terms. Consider, for instance, the case $k=2$. By Legendre's
two-square theorem, $\zeta_{2}(s)$ can be expressed as the product
\[
\zeta_{2}(s)=4\zeta(s)\,L(s,\chi_{4}),
\]
where, for $\text{Re}(s)>1$, $L(s,\chi_{4})=\sum_{k=0}^{\infty}\frac{(-1)^{k}}{(2k+1)^{s}}$
is the usual Dirichlet beta function. It is clear from this expression
that 
\[
\zeta_{2}(s)=\frac{\pi}{s-1}+\pi\gamma+4L^{\prime}\left(1,\chi_{4}\right)+O(s-1).
\]
However, it is well-known that $L^{\prime}(1,\chi_{4})$ can be expressed
in the following form\footnote{see, for instance {[}\cite{vardi}, p. 312{]} for a beautiful proof
of (\ref{L'(1,x) formula}) using the Bohr-Mollerup Theorem.}
\begin{equation}
L^{\prime}\left(1,\chi_{4}\right)=\frac{\gamma}{4}\pi+\frac{\pi}{2}\,\log\left(\frac{2\pi^{\frac{3}{2}}}{\Gamma^{2}\left(\frac{1}{4}\right)}\right),\label{L'(1,x) formula}
\end{equation}
showing the alternative expression for the Laurent series for $\zeta_{2}(s)$,
\begin{equation}
\zeta_{2}(s)=\frac{\pi}{s-1}+2\pi\gamma+2\pi\,\log\left(\frac{2\pi^{\frac{3}{2}}}{\Gamma^{2}\left(\frac{1}{4}\right)}\right)+O(s-1).\label{Meromorphif zeta2(s) with log}
\end{equation}
Alternatively, in order to prove (\ref{Meromorphif zeta2(s) with log})
we can also use the well-known evaluation of Dedekind's $\eta$-function,
\begin{equation}
\eta(i)=\frac{\Gamma\left(\frac{1}{4}\right)}{2\pi^{\frac{3}{4}}},\label{eta(i) ad Remarrrrk}
\end{equation}
which is usually a consequence of Euler's pentagonal theorem.

When $k=4$, an analogous situation takes place. By Jacobi's 4-square
theorem, it is possible to express $\zeta_{4}(s)$ in the following
form
\[
\zeta_{4}(s)=8\left(1-2^{2-2s}\right)\zeta(s)\,\zeta(s-1),
\]
which implies that 
\[
\zeta_{4}(s)=\frac{\pi^{2}}{s-2}+\pi^{2}\gamma+6\zeta^{\prime}(2)+\frac{2}{3}\pi^{2}\log(2)+O(s-2).
\]
On the other hand, we may simplify the previous expansion by invoking a known expression for $\zeta^{\prime}(2)$, i.e., 
\[
\zeta^{\prime}(2)=\frac{\pi^{2}}{6}\left(\gamma+\log(2\pi)-12\log(A)\right),
\]
where $A$ denotes the Glaisher-Kinkelin constant
\begin{equation}
A=\lim_{n\rightarrow\infty}\frac{\prod_{j=1}^{n}j^{j}}{n^{\frac{n^{2}}{2}+\frac{n}{2}+\frac{1}{12}}e^{-\frac{n^{2}}{4}}}. \label{glaisher}
\end{equation}
Thus, one has an alternative Laurent expansion for $\zeta_{4}(s)$ of the form
\begin{equation}
\zeta_{4}(s)=\frac{\pi^{2}}{s-2}+\pi^{2}\left(2\gamma+\log(2\pi)-12\log(A)+\frac{2}{3}\log(2)\right)+O(s-2).\label{Second expression Laurent expansion}
\end{equation}
Comparing (\ref{Second expression Laurent expansion}) and (\ref{Laurent Expansion zeta k at s=00003Dk/2}) with
$k=4$ gives the curious formula 
\begin{equation}
\prod_{(m_{1},m_{2},m_{3})\neq\mathbf{0}}\left(1-e^{-2\pi\sqrt{m_{1}^{2}+m_{2}^{2}+m_{3}^{2}}}\right)=\frac{\exp\left(\frac{\zeta_{3}(2)}{2\pi^{2}}\right)}{2^{\frac{4}{3}}\sqrt{2\pi e}}A^{6}.\label{indeed be an analogue}
\end{equation}
Of course, if it was possible to express $\zeta_{3}(2)$ in terms
of classical constants, we could expect that (\ref{indeed be an analogue})
would indeed be an analogue of the evaluation of $\eta(i)$, (\ref{eta(i) ad Remarrrrk}).

\end{remark}

\section{Proof of Theorem \ref{gen Ferrar paper Petro}}

We start by noting that, in order to prove (\ref{Formula GENERAL Ferrar!}), it suffices to prove the slightly simpler summation formula
\begin{align}
x^{\frac{k}{4}-\frac{1}{2}}\,\sum_{n=1}^{\infty}r_{k}(n)\left\{ \frac{e^{\frac{\pi nx}{2}}}{\sqrt{n}}W_{\frac{1-k}{2},0}\left(\pi nx\right)-\frac{(\pi x)^{\frac{1-k}{2}}}{n^{\frac{k}{2}}}\right\} +\pi^{\frac{1-k}{2}}\,x^{-\frac{k}{4}}\zeta_{k-1}\left(\frac{k}{2}\right)\nonumber \\
+\frac{\sqrt{\pi}x^{-\frac{k}{4}}}{\Gamma\left(\frac{k}{2}\right)}\left(2\gamma+\psi\left(\frac{k}{2}\right)-2\log\left\{ \prod_{\mathbf{m}\in\mathbb{Z}^{k-1}\setminus\{\mathbf{0}\}}\left(1-e^{-2\pi|\mathbf{m}|}\right)\right\} -\log(4\pi x)\right)\nonumber \\
=x^{\frac{1}{2}-\frac{k}{4}}\,\sum_{n=1}^{\infty}r_{k}(n)\left\{ \frac{e^{\frac{\pi n}{2x}}}{\sqrt{n}}W_{\frac{1-k}{2},0}\left(\frac{\pi n}{x}\right)-\left(\frac{\pi}{x}\right)^{\frac{1-k}{2}}\frac{1}{n^{\frac{k}{2}}}\right\} +\pi^{\frac{1-k}{2}}x^{\frac{k}{4}}\zeta_{k-1}\left(\frac{k}{2}\right)\nonumber \\
+\frac{\sqrt{\pi}\,x^{\frac{k}{4}}}{\Gamma\left(\frac{k}{2}\right)}\left(2\gamma+\psi\left(\frac{k}{2}\right)-2\log\left\{ \prod_{\mathbf{m}\in\mathbb{Z}^{k-1}\setminus\{\mathbf{0}\}}\left(1-e^{-2\pi|\mathbf{m}|}\right)\right\} -\log\left(\frac{4\pi}{x}\right)\right),\label{Theorem generalized Ferrar with}
\end{align}
where $x>0$. Clearly, (\ref{Theorem generalized Ferrar with}) is equivalent to (\ref{Formula GENERAL Ferrar!}) when we consider $x=\alpha^{2}$.  
Thus, we may start our argument by considering the infinite series
\begin{equation}
\sum_{n=1}^{\infty}r_{k}(n)\left\{ \frac{e^{\frac{\pi nx}{2}}}{\sqrt{n}}W_{\frac{1-k}{2},0}\left(\pi nx\right)-\frac{(\pi x)^{\frac{1-k}{2}}}{n^{\frac{k}{2}}}\right\}. \label{series at the front}
\end{equation}
Since $r_{k}(n)=O\left(n^{\frac{k}{2}-1}\right)$ and $W_{-\rho,\sigma}(x)$,
$\rho,\sigma\in\mathbb{R}$, has the asymptotic behavior (cf. (\ref{asymp behavior whiutttttakkkkkeeeer rhum}) above)
\[
W_{-\rho,\nu}(x)=x^{-\rho}e^{-\frac{x}{2}}\left\{ 1+O\left(\frac{1}{x}\right)\right\} ,\,\,\,\,\,x\rightarrow\infty,
\]
we can easily conclude the absolute convergence of (\ref{series at the front}). In order to study this first series, we shall consider the useful integral representation {[}\cite{handbook_marichev}, p. 458, relation
(3.30.2.5){]}
\begin{equation}
e^{x/2}W_{-\rho,\sigma}(x)=\frac{1}{2\pi i}\,\intop_{\mu-i\infty}^{\mu+i\infty}\frac{\Gamma\left(s-\sigma+\frac{1}{2}\right)\Gamma\left(s+\sigma+\frac{1}{2}\right)\Gamma\left(-s+\rho\right)}{\Gamma\left(\frac{1}{2}+\rho-\sigma\right)\Gamma\left(\frac{1}{2}+\rho+\sigma\right)}\,x^{-s}ds,\label{formula marichev exp times whittaker in proof of Ferrar!}
\end{equation}
where $|\text{Re}(\sigma)|-\frac{1}{2}<\mu<\text{Re}(\rho)$. If we
shift the line of integration from this range to $\text{Re}(\rho)<\mu^{\prime}<1+\text{Re}(\rho)$,
we find that the integrand in (\ref{formula marichev exp times whittaker in proof of Ferrar!}) has a simple pole located at $s=\rho$.
Thus, by the residue theorem,\footnote{the residue theorem can be trivially applied due to the asymptotic behavior of the integrand with respect to $\text{Im}(s)$, as $|\text{Im}(s)|\rightarrow \infty$. The justification made at (\ref{bound bound bound to justify proof}) is also applicable here.} we have
\begin{equation}
e^{x/2}W_{-\rho,\sigma}(x)-x^{-\rho}=\frac{1}{2\pi i}\,\intop_{\mu^{\prime}-i\infty}^{\mu^{\prime}+i\infty}\frac{\Gamma\left(s-\sigma+\frac{1}{2}\right)\Gamma\left(s+\sigma+\frac{1}{2}\right)\Gamma\left(-s+\rho\right)}{\Gamma\left(\frac{1}{2}+\rho-\sigma\right)\Gamma\left(\frac{1}{2}+\rho+\sigma\right)}\,x^{-s}ds.\label{Mellin barnes integral WHIIITTTAKKKER}
\end{equation}
Taking $\sigma=0$ and $\rho=\frac{k-1}{2}$,
we have that (\ref{Mellin barnes integral WHIIITTTAKKKER}) is valid
in the region $\frac{k-1}{2}<\mu^{\prime}<\frac{k+1}{2}$ and so,
by absolute convergence, the following integral representation holds
\begin{align}
&\sum_{n=1}^{\infty}r_{k}(n)\left\{ \frac{e^{\frac{\pi nx}{2}}}{\sqrt{n}}W_{\frac{1-k}{2},0}\left(\pi nx\right)-\frac{(\pi x)^{\frac{1-k}{2}}}{n^{\frac{k}{2}}}\right\}\nonumber\\
&=\frac{1}{2\pi i}\,\intop_{\mu^{\prime}-i\infty}^{\mu^{\prime}+i\infty}\frac{\Gamma^{2}\left(s+\frac{1}{2}\right)\Gamma\left(\frac{k-1}{2}-s\right)}{\Gamma^{2}\left(\frac{k}{2}\right)}\zeta_{k}\left(s+\frac{1}{2}\right)\,(\pi x)^{-s}ds.\label{starting point Mellin barnes ferrar formula in Rhum paper}
\end{align}
We shift the line of integration in (\ref{starting point Mellin barnes ferrar formula in Rhum paper}) from $\text{Re}(s)=\mu^{\prime}$
to $\text{Re}(s)=\frac{k}{2}-\mu^{\prime}$, taking an integration along a positively oriented rectangular contour, $\mathscr{R}_{\mu^{\prime}}(T)$, with vertices $\mu^{\prime}\pm iT$ and $\frac{k}{2}-\mu^{\prime}\pm iT$, $T>0$. By the Phragm\'en-Lindel\"of principle \cite{titchmarsh_theory_of_functions}, we know that,
for any $\delta>0$, $\zeta_{k}(\sigma+it)\ll_{\delta}|t|^{A(\sigma)+\delta}$,
where $A(\sigma)$ is the function defined by
\begin{equation}
A(\sigma)=\begin{cases}
0 & \sigma>\frac{k}{2}+\delta,\\
\frac{k}{2}-\sigma & -\delta\leq\sigma\leq\frac{k}{2}+\delta,\\
\frac{k}{2}-2\sigma & \sigma<-\delta.
\end{cases} \label{convex estimates cases Popov paper zetak}
\end{equation}
Moreover, by Stirling's formula (\ref{Stirling exact form on Introduction}) and (\ref{convex estimates cases Popov paper zetak}), we know that the integrals along the horizontal segments of $\mathscr{R}_{\mu^{\prime}}(T)$ must satisfy the bound 
\begin{equation}
\intop_{\mu^{\prime}}^{\frac{k}{2}-\mu^{\prime}}\left|T^{A(\sigma)+\frac{k}{2}-1}e^{-\frac{3\pi}{2}T}\,(\pi x)^{-\sigma}\right|d\sigma\ll e^{-\frac{3\pi}{2}T}T^{\frac{k}{2}-1}, \label{bound bound bound to justify proof}
\end{equation}
becoming negligible as $T\rightarrow \infty$.
Since $\frac{k-1}{2}<\mu^{\prime}<\frac{k+1}{2}$
by hypothesis, we know that $\Gamma^{2}\left(s+\frac{1}{2}\right)$
contains no poles inside the rectangular contour $\mathscr{R}_{\mu^{\prime}}(T)$.
Moreover, since $\mu^{\prime}<\frac{k+1}{2}$, the only pole that
$\Gamma\left(\frac{k-1}{2}-s\right)$ contains inside $\mathscr{R}_{\mu^{\prime}}(T)$
is located at $s=\frac{k-1}{2}$ and this is the exact same pole
that $\zeta_{k}\left(s+\frac{1}{2}\right)$ possesses.
The computation of the residue at the point $s=\frac{k-1}{2}$ comes from the Laurent expansion
\[
\Gamma\left(\frac{k-1}{2}-s\right)=\frac{2}{k-1-2s}-\gamma+O\left(s-\frac{k-1}{2}\right),
\]
together with the generalized Kronecker limit formula obtained in
Lemma \ref{continuation Epstein a la Selbegr chowla} above. Straightforward simplifications yield
\begin{align}
&\text{Res}_{s=\frac{k-1}{2}}\left\{ \frac{\Gamma^{2}\left(s+\frac{1}{2}\right)\Gamma\left(\frac{k-1}{2}-s\right)}{\Gamma^{2}\left(\frac{k}{2}\right)}\zeta_{k}\left(s+\frac{1}{2}\right)\,(\pi x)^{-s}\right\}\nonumber \\
&=-\frac{\sqrt{\pi}x^{\frac{1-k}{2}}}{\Gamma\left(\frac{k}{2}\right)}\left(2\gamma+\psi\left(\frac{k}{2}\right)-2\log\left\{ \prod_{\mathbf{m}\in\mathbb{Z}^{k-1}\setminus\{\mathbf{0}\}}\left(1-e^{-2\pi|\mathbf{m}|}\right)\right\} -\log(4\pi x)\right)\nonumber\\
&-(\pi x)^{-\frac{k-1}{2}}\zeta_{k-1}\left(\frac{k}{2}\right).\label{residue in ferrar prof}
\end{align}
Hence, by Cauchy's residue theorem and (\ref{residue in ferrar prof}), we have the equality
\begin{align}
&\frac{1}{2\pi i}\,\intop_{\mu^{\prime}-i\infty}^{\mu^{\prime}+i\infty}\frac{\Gamma^{2}\left(s+\frac{1}{2}\right)\Gamma\left(\frac{k-1}{2}-s\right)}{\Gamma^{2}\left(\frac{k}{2}\right)}\zeta_{k}\left(s+\frac{1}{2}\right)\,(\pi x)^{-s}ds=-(\pi x)^{-\frac{k-1}{2}}\zeta_{k-1}\left(\frac{k}{2}\right)\nonumber\\
&-\frac{\sqrt{\pi}x^{\frac{1-k}{2}}}{\Gamma\left(\frac{k}{2}\right)}\left(2\gamma+\psi\left(\frac{k}{2}\right)-2\log\left\{ \prod_{\mathbf{m}\in\mathbb{Z}^{k-1}\setminus\{\mathbf{0}\}}\left(1-e^{-2\pi|\mathbf{m}|}\right)\right\} -\log(4\pi x)\right)\nonumber\\
&+\frac{1}{2\pi i}\,\intop_{\frac{k}{2}-\mu^{\prime}-i\infty}^{\frac{k}{2}-\mu^{\prime}+i\infty}\frac{\Gamma^{2}\left(s+\frac{1}{2}\right)\Gamma\left(\frac{k-1}{2}-s\right)}{\Gamma^{2}\left(\frac{k}{2}\right)}\zeta_{k}\left(s+\frac{1}{2}\right)\,(\pi x)^{-s}ds.\label{after cauchy proof ferrar general}
\end{align}
Our proof will be complete once we evaluate the last integral in (\ref{after cauchy proof ferrar general}). In order to do this, we appeal to the functional
equation for $\zeta_{k}(s)$, (\ref{functional equation zetak}), and make the change of variables $s\leftrightarrow\frac{k}{2}-s$, resulting in 
\begin{align}
&\frac{1}{2\pi i}\,\intop_{\frac{k}{2}-\mu^{\prime}-i\infty}^{\frac{k}{2}-\mu^{\prime}+i\infty}\frac{\Gamma^{2}\left(s+\frac{1}{2}\right)\Gamma\left(\frac{k-1}{2}-s\right)}{\Gamma^{2}\left(\frac{k}{2}\right)}\zeta_{k}\left(s+\frac{1}{2}\right)\,(\pi x)^{-s}ds\nonumber \\
&=\frac{\pi}{2\pi i}\,\intop_{\mu^{\prime}-i\infty}^{\mu^{\prime}+i\infty}\frac{\Gamma\left(\frac{k+1}{2}-s\right)\Gamma\left(s-\frac{1}{2}\right)}{\Gamma^{2}\left(\frac{k}{2}\right)}\pi^{-s}\,\Gamma\left(s-\frac{1}{2}\right)\zeta_{k}\left(s-\frac{1}{2}\right)\,x^{s-\frac{k}{2}}\,ds\nonumber \\
&=\frac{x^{1-\frac{k}{2}}}{2\pi i}\,\intop_{\mu^{\prime}-1-i\infty}^{\mu^{\prime}-1+i\infty}\frac{\Gamma^{2}\left(s+\frac{1}{2}\right)\Gamma\left(\frac{k-1}{2}-s\right)}{\Gamma^{2}\left(\frac{k}{2}\right)}\,\zeta_{k}\left(s+\frac{1}{2}\right)\left(\frac{\pi}{x}\right)^{-s}\,ds,\label{integral ferrar ginal evaluion}
\end{align}
where $\frac{k-3}{2}<\mu^{\prime}-1<\frac{k-1}{2}$. It would be highly
convenient to use (\ref{starting point Mellin barnes ferrar formula in Rhum paper})
in order to evaluate the last integral in (\ref{integral ferrar ginal evaluion}).
But to be able to do this requires that we shift the line of integration from
$\text{Re}(s)=\mu^{\prime}-1$ back to $\text{Re}(s)=\mu^{\prime}$. To make this procedure, we need to apply the residue theorem again. Using the calculations and technical justifications made in (\ref{residue in ferrar prof}), we easily find that 
\begin{align*}
&\text{Res}_{s=\frac{k-1}{2}}\left\{ \frac{\Gamma^{2}\left(s+\frac{1}{2}\right)\Gamma\left(\frac{k-1}{2}-s\right)}{\Gamma^{2}\left(\frac{k}{2}\right)}\,\zeta_{k}\left(s+\frac{1}{2}\right)\left(\frac{\pi}{x}\right)^{-s}\right\} \\
&=-\frac{\sqrt{\pi}x^{\frac{k-1}{2}}}{\Gamma\left(\frac{k}{2}\right)}\left(2\gamma+\psi\left(\frac{k}{2}\right)-2\log\left\{ \prod_{\mathbf{m}\in\mathbb{Z}^{k-1}\setminus\{\mathbf{0}\}}\left(1-e^{-2\pi|\mathbf{m}|}\right)\right\} -\log\left(\frac{4\pi}{x}\right)\right)\\
&-\left(\frac{\pi}{x}\right)^{-\frac{k-1}{2}}\zeta_{k-1}\left(\frac{k}{2}\right),
\end{align*}
and so the last member of (\ref{integral ferrar ginal evaluion}) can be written as 
\begin{align}
&\frac{x^{1-\frac{k}{2}}}{2\pi i}\,\intop_{\mu^{\prime}-1-i\infty}^{\mu^{\prime}-1+i\infty}\frac{\Gamma^{2}\left(s+\frac{1}{2}\right)\Gamma\left(\frac{k-1}{2}-s\right)}{\Gamma^{2}\left(\frac{k}{2}\right)}\,\zeta_{k}\left(s+\frac{1}{2}\right)\left(\frac{\pi}{x}\right)^{-s}\,ds\nonumber\\
&=\frac{x^{1-\frac{k}{2}}}{2\pi i}\,\intop_{\mu^{\prime}-i\infty}^{\mu^{\prime}+i\infty}\frac{\Gamma^{2}\left(s+\frac{1}{2}\right)\Gamma\left(\frac{k-1}{2}-s\right)}{\Gamma^{2}\left(\frac{k}{2}\right)}\,\zeta_{k}\left(s+\frac{1}{2}\right)\left(\frac{\pi}{x}\right)^{-s}\,ds+\sqrt{x}\pi^{\frac{1-k}{2}}\zeta_{k-1}\left(\frac{k}{2}\right)\nonumber\\
&+\frac{\sqrt{\pi x}}{\Gamma\left(\frac{k}{2}\right)}\left(2\gamma+\psi\left(\frac{k}{2}\right)-2\log\left\{ \prod_{\mathbf{m}\in\mathbb{Z}^{k-1}\setminus\{\mathbf{0}\}}\left(1-e^{-2\pi|\mathbf{m}|}\right)\right\} -\log\left(\frac{4\pi}{x}\right)\right).\label{final nail in the proof of ferrar general}
\end{align}
The summation formula (\ref{Theorem generalized Ferrar with}) now results from a combination of formulas (\ref{starting point Mellin barnes ferrar formula in Rhum paper})-(\ref{final nail in the proof of ferrar general}). $\blacksquare$

\begin{corollary}\label{corollary as ferrar once afgain}
Let $\alpha,\beta>0$ be such that $\alpha\beta=1$. Then Ferrar's formula (\ref{ferrar introduction section}) holds, this is, 
\begin{align}
2\sqrt{\alpha}\,\sum_{n=1}^{\infty}\left\{ e^{\frac{\pi n^{2}\alpha^{2}}{2}}K_{0}\left(\frac{\pi n^{2}\alpha^{2}}{2}\right)-\frac{1}{\alpha n}\right\} +\frac{1}{\sqrt{\alpha}}\left(\gamma-\log(16\pi)-2\log\left(\alpha\right)\right)\nonumber \\
=2\sqrt{\beta}\,\sum_{n=1}^{\infty}\left\{ e^{\frac{\pi n^{2}\beta^{2}}{2}}K_{0}\left(\frac{\pi n^{2}\beta^{2}}{2}\right)-\frac{1}{\beta n}\right\} +\frac{1}{\sqrt{\beta}}\left(\gamma-\log(16\pi)-2\log\left(\beta\right)\right).\label{ferrar formula in RUM paper}
\end{align}
\end{corollary}

\begin{proof}
We set $k=1$ in (\ref{Formula GENERAL Ferrar!}).
Once again, we recall the conventions that $\Phi_{1}(y):=1$ and
$\zeta_{0}(s)=0$ (see Remark \ref{trivial remark when k=1} above). Remembering also that $r_{1}(n)$ is $2$ if $n$ is
a perfect square and is zero otherwise and, finally, appealing to the
special value for Euler's digamma function, $\psi\left(\frac{1}{2}\right)=-2\log(2)-\gamma$, we can deduce the transformation formula
\begin{align}
\frac{2}{\sqrt{\alpha}}\,\sum_{n=1}^{\infty}\left\{ \frac{e^{\frac{\pi n^{2}\alpha^{2}}{2}}}{n}W_{0,0}\left(\pi n^{2}\alpha^{2}\right)-\frac{1}{n}\right\} +\frac{1}{\sqrt{\alpha}}\left(\gamma-\log(16\pi)-2\log\left(\alpha\right)\right)\nonumber \\
=\frac{2}{\sqrt{\beta}}\,\sum_{n=1}^{\infty}\left\{ \frac{e^{\frac{\pi n^{2}\beta^{2}}{2}}}{n}W_{0,0}\left(\pi n^{2}\beta^{2}\right)-\frac{1}{n}\right\} +\frac{1}{\sqrt{\beta}}\left(\gamma-\log(16\pi)-2\log\left(\beta\right)\right).\label{W00 formula}
\end{align}
Using the reduction formula (\ref{Whittaker to MAcdonald RHUM}),
\[
W_{0,0}(x)=\sqrt{\frac{x}{\pi}}\,K_{0}\left(\frac{x}{2}\right),
\]
we find that (\ref{W00 formula}) yields Ferrar's identity (\ref{ferrar formula in RUM paper}).
\end{proof}

\begin{corollary}
Let $\alpha,\beta>0$ be such that $\alpha\beta=1$. Then the following
summation formula is valid
\begin{align}
\alpha\sum_{n=1}^{\infty}r_{2}(n)\left\{ e^{\pi n\alpha^{2}}\,\text{Ei}\left(-\pi n\alpha^{2}\right)+\frac{1}{\pi\alpha^{2}n}\right\} -\frac{1}{\alpha}\left(\gamma-4\log\left(\eta(i)\right)-\log\left(4\pi\alpha^{2}\right)\right)\nonumber \\
=\beta\sum_{n=1}^{\infty}r_{2}(n)\left\{ e^{\pi n\beta^{2}}\text{Ei}\left(-\pi n\beta^{2}\right)+\frac{1}{\pi\beta^{2}n}\right\} - \frac{1}{\beta}\left(\gamma-4\log\left(\eta(i)\right)-\log\left(4\pi\beta^{2}\right)\right),\label{Ferrar for r2(n) RHUM}
\end{align}
where, for $\text{Im}(\tau)>0$, $\eta(\tau)$ denotes Dedekind's
$\eta-$function (\ref{dedekind definition" RHUUUUM}) and $\text{Ei}(-x)$ represents the exponential integral function
\begin{equation}
\text{Ei}(-x):=-E_{1}(x)=-\intop_{x}^{\infty}\frac{e^{-u}}{u}\,du. \label{definition Ei in Ferrar}
\end{equation}
\end{corollary}

\begin{proof}
We set $k=2$ in (\ref{Formula GENERAL Ferrar!})
and use the fact that $\Phi_{2}(y)$ is connected to $\eta(iy)$ by (\ref{Phi 2 and connection to eta Dedekind}). This procedure yields the formula
\begin{align}
\sum_{n=1}^{\infty}r_{2}(n)\left\{ \frac{e^{\frac{\pi n\alpha^{2}}{2}}}{\sqrt{n}}W_{-\frac{1}{2},0}\left(\pi n\alpha^{2}\right)-\frac{1}{\sqrt{\pi}\,\alpha n}\right\} +\frac{\sqrt{\pi}}{\alpha}\left(\gamma-4\log\left(\eta(i)\right)-\log\left(4\pi\alpha^{2}\right)\right)\nonumber \\
=\sum_{n=1}^{\infty}r_{2}(n)\left\{ \frac{e^{\frac{\pi n\beta^{2}}{2}}}{\sqrt{n}}W_{-\frac{1}{2},0}\left(\pi n\beta^{2}\right)-\frac{1}{\sqrt{\pi}\,\beta n}\right\} +\frac{\sqrt{\pi}}{\beta}\left(\gamma-4\log\left(\eta(i)\right)-\log\left(4\pi\beta^{2}\right)\right),
\end{align}
Next, comparing the Mellin-Barnes integral (\ref{formula marichev exp times whittaker in proof of Ferrar!}) with [\cite{handbook_marichev}, p. 102, relation 3.3.2.1], 
\[
e^{x}\text{Ei}(-x)=-\frac{1}{2\pi i}\intop_{\sigma-i\infty}^{\sigma+i\infty}\Gamma^{2}(s)\,\Gamma(1-s)\,x^{-s}ds,\,\,\,\,\,0<\sigma<1,
\]
we find out that, for any $x>0$,
\begin{equation}
W_{-\frac{1}{2},0}(x)=-\sqrt{x}e^{\frac{x}{2}}\,\text{Ei}(-x), \label{Whittaker -1/2 0}
\end{equation}
which completes the proof of (\ref{Ferrar for r2(n) RHUM}). 
\end{proof}

\begin{remark}
Using the evaluation for $\eta(i)$ presented in (\ref{eta(i) ad Remarrrrk}), 
we can rewrite the two-dimensional analogue of Ferrar's formula (\ref{Ferrar for r2(n) RHUM})
in the following elegant manner
\begin{align}
\alpha\sum_{n=1}^{\infty}r_{2}(n)\left\{e^{\pi n\alpha^{2}}\,\text{Ei}\left(-\pi n\alpha^{2}\right)+\frac{1}{\pi\alpha^{2}n}\right\} -\frac{1}{\alpha}\left(\gamma-2\log\left(\frac{\Gamma^{2}\left(\frac{1}{4}\right)}{2\pi}\alpha\right)\right)\nonumber \\
=\beta\sum_{n=1}^{\infty}r_{2}(n)\left\{ e^{\pi n\beta^{2}}\text{Ei}\left(-\pi n\beta^{2}\right)+\frac{1}{\pi\beta^{2}n}\right\} - \frac{1}{\beta}\left(\gamma-2\log\left(\frac{\Gamma^{2}\left(\frac{1}{4}\right)}{2\pi}\beta\right)\right),\label{with explicit evalation of eta(i) FERRAR}
\end{align}
where $\alpha,\beta>0$ are such that $\alpha\beta=1$.
Using the definition of the arithmetical function $r_{2}(n)$, we
can rewrite both infinite series of (\ref{Ferrar for r2(n) RHUM})
as double series containing two variables of summation, resulting in a transformation formula of the form
\begin{align*}
\alpha\sum_{m,n\neq0}\left\{e^{\pi\alpha^{2}\left(m^{2}+n^{2}\right)}\,\text{Ei}\left(-\pi\alpha^{2}\left(m^{2}+n^{2}\right)\right)+\frac{1}{\pi\alpha^{2}\left(m^{2}+n^{2}\right)}\right\} -\frac{1}{\alpha}\left(\gamma-2\log\left(\frac{\Gamma^{2}\left(\frac{1}{4}\right)}{2\pi}\alpha\right)\right)\\
=\beta\sum_{m,n\neq0}\left\{e^{\pi\beta^{2}\left(m^{2}+n^{2}\right)}\text{Ei}\left(-\pi\beta^{2}\left(m^{2}+n^{2}\right)\right)+\frac{1}{\pi\beta^{2}\left(m^{2}+n^{2}\right)}\right\} - \frac{1}{\beta}\left(\gamma-2\log\left(\frac{\Gamma^{2}\left(\frac{1}{4}\right)}{2\pi}\beta\right)\right).
\end{align*}
\end{remark}

\begin{remark}
Other interesting examples can be deduced. For example, using the
reduction formula for the Whittaker function
\[
W_{-1,0}(x)=\frac{2\sqrt{x}(x+1)}{\sqrt{\pi}}K_{0}\left(\frac{x}{2}\right)-\frac{2x^{\frac{3}{2}}}{\sqrt{\pi}}K_{1}\left(\frac{x}{2}\right),
\]
one can write down yet another particular case of Ferrar's formula (\ref{Theorem generalized Ferrar with})
for the arithmetical function $r_{3}(n)$, namely, 
\begin{align}
\sqrt{\alpha}\,\sum_{n=1}^{\infty}r_{3}(n)\,\left\{ 2\pi n\alpha^{3}e^{\frac{\pi n\alpha^{2}}{2}}\left\{ K_{0}\left(\frac{\pi n\alpha^{2}}{2}\right)-K_{1}\left(\frac{\pi n\alpha^{2}}{2}\right)\right\} +2\alpha e^{\frac{\pi n\alpha^{2}}{2}}K_{0}\left(\frac{\pi n\alpha^{2}}{2}\right)-\frac{1}{\pi\alpha^{2}n^{\frac{3}{2}}}\right\} 	\nonumber
\\
+\frac{4}{\pi\alpha^{\frac{3}{2}}}\zeta\left(\frac{3}{2}\right)\,L\left(\frac{3}{2},\chi_{4}\right)+\frac{2}{\alpha^{\frac{3}{2}}}\left(\gamma+2-2\log\left\{ 2\prod_{m_{1},m_{2}\neq0}\left(1-e^{-2\pi\sqrt{m_{1}^{2}+m_{2}^{2}}}\right)\right\} -\log\left(4\pi\alpha^{2}\right)\right) \nonumber \\
=\sqrt{\beta}\,\sum_{n=1}^{\infty}r_{3}(n)\,\left\{ 2\pi n\beta^{3}e^{\frac{\pi n\beta^{2}}{2}}\left\{ K_{0}\left(\frac{\pi n\beta^{2}}{2}\right)-K_{1}\left(\frac{\pi n\beta^{2}}{2}\right)\right\} +2\beta e^{\frac{\pi n\beta^{2}}{2}}K_{0}\left(\frac{\pi n\beta^{2}}{2}\right)-\frac{1}{\pi\beta^{2}n^{\frac{3}{2}}}\right\} \nonumber \\
+\frac{4}{\pi\beta^{\frac{3}{2}}}\zeta\left(\frac{3}{2}\right)\,L\left(\frac{3}{2},\chi_{4}\right)+\frac{2}{\beta^{\frac{3}{2}}}\left(\gamma+2-2\log\left\{ 2\prod_{m_{1},m_{2}\neq0}\left(1-e^{-2\pi\sqrt{m_{1}^{2}+m_{2}^{2}}}\right)\right\} -\log\left(4\pi\beta^{2}\right)\right). \label{proof r3(n) Ferrar}
\end{align}

It is noteworthy that Ferrar could have derived (\ref{proof r3(n) Ferrar}) using only the functional equation for $\zeta_{3}(s)$ (\ref{functional equation zetak}) and the meromorphic expansion (\ref{Laurent Expansion zeta k at s=00003Dk/2}) for $k=3$. His knowledge of Popov's formula (\ref{ferrar_berndt_sum_of_squares_VTEX1}) for the case $k=2$, which he and Dixon had previously obtained as a corollary of Vorono\"i's summation formula [\cite{dixon_ferrar_circle(i)}, p. 51, eq. (3.12)] (see (\ref{III-106}) above), would have provided him the necessary tools. Similarly, he could have derived (\ref{with explicit evalation of eta(i) FERRAR}) by relying on Watson's formula (\ref{watson_intro_book}).

\end{remark}

\section{Proof of Theorem \ref{Koshliakov theorem rk(n)}}

Take $\alpha=\pi x$: then the formula we aim to prove, i.e. (\ref{general_KOSH_rk(n)}), is equivalent to the
relation
\begin{align}
2\,\sum_{m,n=1}^{\infty}r_{k}(m)\,r_{k}(n)\,K_{0}\left(2\pi\sqrt{m\,n}\,x\right)-\frac{2}{x^{k}}\,\sum_{m,n=1}^{\infty}r_{k}(m)\,r_{k}(n)\,K_{0}\left(\frac{2\pi\sqrt{m\,n}}{x}\right)\nonumber \\
=x^{-k}\left\{ 2\pi^{-\frac{k}{2}}\Gamma\left(\frac{k}{2}\right)\zeta_{k-1}\left(\frac{k}{2}\right)+2\gamma-4\log\left\{ \prod_{\mathbf{m}\in\mathbb{Z}^{k-1}\setminus\{\mathbf{0}\}}\left(1-e^{-2\pi|\mathbf{m}|}\right)\right\} -2\log\left(4\pi x\right)\right\} \nonumber \\
-2\pi^{-\frac{k}{2}}\Gamma\left(\frac{k}{2}\right)\zeta_{k-1}\left(\frac{k}{2}\right)-2\gamma+4\log\left\{ \prod_{\mathbf{m}\in\mathbb{Z}^{k-1}\setminus\{\mathbf{0}\}}\left(1-e^{-2\pi|\mathbf{m}|}\right)\right\} -2\log\left(\frac{x}{4\pi}\right).\label{Preliminary step towards proof}
\end{align}
We will prove (\ref{Preliminary step towards proof}) by starting with the infinite series
\[
\sum_{m,n=1}^{\infty}r_{k}(m)\,r_{k}(n)\,K_{0}\left(2\pi\sqrt{m\,n}\,x\right),
\]
which can be analytically treated by appealing to the Mellin-Barnes integral {[}\cite{handbook_marichev}, p. 204,
eq. (3.14.1.3){]}
\[
K_{0}(x)=\frac{1}{8\pi i}\,\intop_{\mu-i\infty}^{\mu+i\infty}\Gamma^{2}\left(\frac{s}{2}\right)\,\left(\frac{x}{2}\right)^{-s}ds,
\]
where $\mu>0$. Taking $\mu>k$, we see that
\begin{align}
\sum_{m,n=1}^{\infty}r_{k}(m)\,r_{k}(n)\,K_{0}\left(2\pi\sqrt{m\,n}\,x\right) & =\frac{1}{8\pi i}\,\intop_{\mu-i\infty}^{\mu+i\infty}\pi^{-s}\,\Gamma^{2}\left(\frac{s}{2}\right)\zeta_{k}^{2}\left(\frac{s}{2}\right)\,x^{-s}\,ds\nonumber \\
 & :=\frac{1}{4\pi i}\,\intop_{\frac{\mu}{2}-i\infty}^{\frac{\mu}{2}+i\infty}\eta_{k}^{2}(s)\,x^{-2s}\,ds,\label{Starting point Koshliakov generaaaalll!}
\end{align}
where the interchange of the orders of summation and integration is
just a consequence of Stirling's formula (\ref{Stirling exact form on Introduction}) and the choice $\mu>k$.
Furthermore, $\eta_{k}(s)$ denotes the completed zeta function defined
in (\ref{def eta_k}). We shall move the line of integration from $\text{Re}(s)=\frac{\mu}{2}$
to $\text{Re}(s)=-\delta$, $0<\delta<1$, by integrating along a
positively oriented rectangular contour $\mathscr{R}_{\mu}(T)$ with vertices $\frac{\mu}{2}\pm iT$
and $-\delta\pm iT$, $T>0$. We find that the integrand in (\ref{Starting point Koshliakov generaaaalll!})
possesses two double poles located at $\mathscr{R}_{\mu}(T)$, each one
located at the points $s=\frac{k}{2}$ and $s=0$. We may compute
the residue at $s=\frac{k}{2}$ by using Lemma \ref{continuation Epstein a la Selbegr chowla}, which gives
\[
\mathscr{R}_{\frac{k}{2}}:=\text{Res}_{s=\frac{k}{2}}\left\{ \eta_{k}^{2}(s)\,x^{-2s}\right\} =2x^{-\frac{k}{2}}\left\{ \pi^{-\frac{k}{2}}\Gamma\left(\frac{k}{2}\right)\mathscr{C}_{0}+\psi\left(\frac{k}{2}\right)-\log\left(\pi x\right)\right\} ,
\]
where $\mathscr{C}_{0}$ denotes the constant term in the Laurent
expansion for $\zeta_{k}(s)$, (\ref{Laurent Expansion zeta k at s=00003Dk/2}). Replacing $\mathscr{C}_{0}$ by its complete expression,
we find that
\begin{align}
\mathscr{R}_{\frac{k}{2}} & =2x^{-\frac{k}{2}}\left\{ \pi^{-\frac{k}{2}}\Gamma\left(\frac{k}{2}\right)\mathscr{C}_{0}+\psi\left(\frac{k}{2}\right)-\log\left(\pi x\right)\right\} \nonumber \\
=x^{-\frac{k}{2}} & \left\{ 2\pi^{-\frac{k}{2}}\Gamma\left(\frac{k}{2}\right)\,\zeta_{k-1}\left(\frac{k}{2}\right)+2\gamma-4\log\left\{ \prod_{\mathbf{m}\in\mathbb{Z}^{k-1}\setminus\{\mathbf{0}\}}\left(1-e^{-2\pi|\mathbf{m}|}\right)\right\} -2\log\left(4\pi x\right)\right\}\nonumber \\
=x^{-\frac{k}{2}} & \left\{ 2\pi^{-\frac{k}{2}}\Gamma\left(\frac{k}{2}\right)\,\zeta_{k-1}\left(\frac{k}{2}\right)+2\gamma-4\log\left(\Phi_{k}(1)\right)-2\log\left(4\pi x\right)\right\},
\label{computation R_k/2}
\end{align}
where $\Phi_{k}(y)$ denotes the analogue of the square of Dedekind's $\eta-$function defined by (\ref{generalization Dedekind}). 
Using the functional equation for $\zeta_{k}(s)$, (\ref{functional equation zetak}), and the computation of $\mathscr{R}_{\frac{k}{2}}$ (\ref{computation R_k/2}), we see that  
\begin{align*}
\mathscr{R}_{0}:=\text{Res}_{s=0}\left\{ \eta_{k}^{2}(s)\,x^{-2s}\right\}  & =-x^{-\frac{k}{2}}\text{Res}_{s=\frac{k}{2}}\left\{ \eta_{k}^{2}(s)\left(\frac{1}{x}\right)^{-2s}\right\} \\
=-2\pi^{-\frac{k}{2}}\Gamma\left(\frac{k}{2}\right)\,\zeta_{k-1}\left(\frac{k}{2}\right) & -2\gamma+4\,\log\left\{ \prod_{\mathbf{m}\in\mathbb{Z}^{k-1}\setminus\{\mathbf{0}\}}\left(1-e^{-2\pi|\mathbf{m}|}\right)\right\} -2\log\left(\frac{x}{4\pi}\right)\\
=-2\pi^{-\frac{k}{2}}\Gamma\left(\frac{k}{2}\right)\,\zeta_{k-1}\left(\frac{k}{2}\right) & -2\gamma+4\,\log\left(\Phi_{k}(1)\right)-2\log\left(\frac{x}{4\pi}\right).
\end{align*}
By Stirling's formula (\ref{Stirling exact form on Introduction}) and the Phragm\'en-Lindel\"of principle for $\zeta_{k}(s)$
(\ref{convex estimates cases Popov paper zetak}), we know that 
\[
\intop_{-\delta}^{\mu/2}\,\left|\eta_{k}^{2}\left(\sigma+iT\right)\right|x^{-2\sigma}d\sigma\ll e^{-\pi T}\,\intop_{-\delta}^{\mu/2}T^{2A(\sigma)}x^{-2\sigma}d\sigma\ll T^{B}e^{-\pi T},
\]
and so the integrals along the horizontal
segments of the rectangle $\mathscr{R}_{\mu}(T)$ tend to zero as $T\rightarrow\infty$. Hence, a
combination of Cauchy's residue theorem, the functional equation (\ref{functional equation zetak})
and (\ref{Starting point Koshliakov generaaaalll!}) gives
\begin{align*}
2\sum_{m,n=1}^{\infty}r_{k}(m)\,r_{k}(n)\,K_{0}\left(2\pi\sqrt{m\,n}\,x\right) & =2\mathscr{R}_{\frac{k}{2}}+2\mathscr{R}_{0}+\frac{1}{2\pi i}\,\intop_{-\delta-i\infty}^{-\delta+i\infty}\eta_{k}^{2}(s)\,x^{-2s}\,ds\\
 & =2\mathscr{R}_{\frac{k}{2}}+2\mathscr{R}_{0}+\frac{x^{-k}}{2\pi i}\,\intop_{\frac{k}{2}+\delta-i\infty}^{\frac{k}{2}+\delta+i\infty}\eta_{k}^{2}(s)\,x^{2s}\,ds\\
 & =2\mathscr{R}_{\frac{k}{2}}+2\mathscr{R}_{0}+\frac{2}{x^{k}}\,\sum_{m,n=1}^{\infty}r_{k}(m)\,r_{k}(n)\,K_{0}\left(\frac{2\pi\sqrt{m\,n}}{x}\right),
\end{align*}
which completes the proof of (\ref{Preliminary step towards proof}). $\blacksquare$

\bigskip{}

\bigskip{}

We end this paper with two very interesting corollaries that can be obtained from Theorem \ref{Koshliakov theorem rk(n)} when $k=2$ and $k=4$. To present the first of these corollaries, recall from Remark \ref{analogues infinite product lemma} that $\Phi_{2}(1)=e^{\frac{\pi}{6}}\eta(i)^{2}$ (see also (\ref{Phi 2 and connection to eta Dedekind})). With this in mind, the following corollary is quite easy to check. 

\begin{corollary}
Let $\alpha,\beta$ be two positive real numbers such that $\alpha\beta=\pi^{2}$.
Then the following analogue of Koshliakov's formula holds
\begin{align*}
\alpha\left\{ \gamma-4\log\left(\frac{\Gamma\left(\frac{1}{4}\right)}{2\pi^{\frac{3}{4}}}\right)-\log\left(4\beta\right)+\sum_{m,n=1}^{\infty}r_{2}(m)\,r_{2}(n)\,K_{0}\left(2\sqrt{m\,n}\,\alpha\right)\right\} \\
=\beta\left\{ \gamma-4\log\left(\frac{\Gamma\left(\frac{1}{4}\right)}{2\pi^{\frac{3}{4}}}\right)-\log\left(4\alpha\right)+\sum_{m,n=1}^{\infty}r_{2}(m)\,r_{2}(n)\,K_{0}\left(2\sqrt{m\,n}\beta\right)\right\} .
\end{align*}
\end{corollary}

Analogously, since (\ref{Second expression Laurent expansion}) offers a Laurent expansion for the Epstein zeta function $\zeta_{4}(s)$ containing the Glaisher-Kinkelin constant, the following interesting summation formula can be stated as an immediate consequence of Theorem \ref{Koshliakov theorem rk(n)}. 

\begin{corollary}
Let $\alpha,\beta$ be two positive real numbers such that $\alpha\beta=\pi^{2}$.
Then the following summation formula holds
\begin{align*}
\alpha^{2}\left\{ 2\gamma-24\log(A)+\frac{16}{3}\log(2)+2-2\log\left(\frac{2\beta}{\pi}\right)+2\sum_{m,n=1}^{\infty}r_{4}(m)\,r_{4}(n)\,K_{0}\left(2\sqrt{m\,n}\,\alpha\right)\right\} \\
=\beta^{2}\left\{ 2\gamma-24\log(A)+\frac{16}{3}\log(2)+2-2\log\left(\frac{2\alpha}{\pi}\right)+2\sum_{m,n=1}^{\infty}r_{4}(m)\,r_{4}(n)\,K_{0}\left(2\sqrt{m\,n}\,\beta\right)\right\} ,
\end{align*}
where $A$ is the Glaisher-Kinkelin constant (\ref{glaisher}).
\end{corollary}

\textit{Acknowledgements:} As this is a tribute publication, I shall repeat the words I wrote in the \textit{Acknowledgment} section of my PhD thesis. Therefore, I am deeply grateful to my advisor and, I would gladly say, dear friend,
Semyon Yakubovich, to whom this volume is dedicated. His unwavering support and guidance throughout
the hectic years of my PhD have been a constant source of strength. It is
difficult to imagine a more exemplary mentor. Thank you, Semyon, for embodying
hard work, courage, and kindness, and for the countless enriching
discussions about our beloved topics inside Mathematics. Without you,
I would not have found the courage to begin, let alone complete, the work that led to my PhD.

\end{document}